\documentclass[11pt]{amsart}

\usepackage{amsmath, amsthm}
\usepackage{eucal}
\usepackage{palatino}
\usepackage{euler}
\usepackage{mathrsfs}
\usepackage{amssymb}
\usepackage{amscd}
\usepackage{latexsym}
\usepackage{epsfig}
\usepackage{graphicx}
\usepackage{amsfonts}
\usepackage{psfrag}
\usepackage{caption}
\usepackage{rotating}
\usepackage{mathtools}
\usepackage{fullpage}

\usepackage[normalem]{ulem}

\usepackage{pifont}
\usepackage{epsfig}

\usepackage{url}
\usepackage{cite}
\usepackage{dsfont}
\usepackage{array}

\input xy
\xyoption{all}
\UseComputerModernTips

\numberwithin{equation}{section}
\numberwithin{figure}{section}
\numberwithin{table}{section}

\newtheorem{theorem}[equation]{Theorem}

\newtheorem{lemma}[equation]{Lemma}
\newtheorem{proposition}[equation]{Proposition}
\newtheorem{cor}[equation]{Corollary}

\theoremstyle{definition}
\newtheorem{definition}[equation]{Definition}
\newtheorem{remark}[equation]{Remark}

\newtheorem{example}[equation]{Example}

\makeatletter
\let\c@equation\c@figure
\makeatother

\makeatletter
\let\c@table\c@figure
\makeatother

\newcommand{\C}{{\mathbb{C}}}
\newcommand{\Z}{{\mathbb{Z}}}
\newcommand{\Q}{{\mathbb{Q}}}
\newcommand{\R}{{\mathbb{R}}}
\newcommand{\T}{{\mathbb{T}}}
\renewcommand{\SS}{{\mathbb{S}}}
\newcommand{\id}{\mathds{1}}
\newcommand{\X}{{\mathcal{X}}}
\renewcommand{\to}{\longrightarrow}
\newcommand{\algt}{\mathfrak{t}}
\renewcommand{\iff}{\Leftrightarrow}
\newcommand{\into}{\hookrightarrow}
\newcommand{\FreeProd}{\mathop{\scalebox{1.75}{\raisebox{0ex}{$\ast$}}}}

\DeclareMathOperator{\rank}{rank}
\DeclareMathOperator{\coker}{coker}
\DeclareMathOperator{\spanspan}{span}

\begin{document}

\title[The fundamental group and Betti numbers of toric origami manifolds]{The fundamental group and Betti numbers\\ of toric origami manifolds}

\author{Tara S. Holm}
\thanks{Tara Holm was partially supported by Grant \#208975 from the Simons Foundation and NSF Grant DMS--1206466.}
\address{Department of Mathematics, Malott Hall, Cornell
  University, Ithaca, New York 14853-4201, USA}
\email{tsh@math.cornell.edu}
\urladdr{\url{http://www.math.cornell.edu/~tsh/}}

\author{Ana Rita Pires}
\address{Department of Mathematics, Malott Hall, Cornell
  University, Ithaca, New York 14853-4201, USA}
\email{apires@math.cornell.edu}
\urladdr{\url{http://www.math.cornell.edu/~apires/}}

\keywords{toric symplectic manifold, toric origami manifold,
Delzant polytope, origami template,
fundamental group,
Betti numbers,
cohomology}

\subjclass[2010]{Primary: 53D20; Secondary: 55N91, 57R91}

\date{\today}


\begin{abstract}
Toric origami manifolds are characterized by origami templates, which are combinatorial
models built by gluing polytopes together along facets.  In this paper, we examine the topology
of orientable toric oigami manifolds with co\"orientable folding hypersurface.  We determine the
fundamental group.    In our previous paper \cite{holm-pires}, we studied the ordinary and equivariant
cohomology rings of simply connected  toric origami manifolds.  We conclude this paper by computing
some Betti numbers in the non-simply connected case.
\end{abstract}

\maketitle

\setcounter{tocdepth}{1}
\tableofcontents

\renewcommand{\arraystretch}{1.2}

\section*{Introduction}\label{sec:intro}

Smooth toric varieties and their generalizations are manifolds whose geometry and topology
can be characterized by combinatorial models.  The interplay between geometry and topology
on the one hand and algebra, combinatorics, and discrete geometry on the other has
been integral to our understanding of toric varieties.  In this paper, we study toric origami
manifolds, a class of toric manifolds that arise in symplectic geometry.  The geometry
of toric origami manifolds is encoded in an origami template: a collection of (equi-dimensional)
polytopes with certain facets identified.  In our previous paper \cite{holm-pires}, we studied
the simplest examples of toric origami manifolds, the acyclic ones.  In this manuscript, we develop
new techniques to address the complications that arise in the cyclic case.

We first study the fundamental group of a toric origami manifold.  Building on work
of Masuda and Park \cite{masuda-park} and others, we use the combinatorics of the
origami template to determine the fundamental group of a toric origami manifold 
(Theorem~\ref{thm:fund group}).  
The key trick is to build a simply connected cover of the origami template.
As a consequence of our result, we may deduce that
a toric origami manifold is simply connected if and only if it is acyclic.
We can use our result (namely, the form of the fundamental group)
to show the existence of a $4$-dimensional manifold equipped
with an effective $\T^2$ action which is not a toric origami manifold (Remark~\ref{rmk:type L}).
We then turn to the Betti numbers of a toric origami manifold.  When $M$ is orientable, there
is a natural decomposition $M=M_+\cup M_-$, where $M_+\cap M_-\cong Z$ is the folding
hypersurface.   There are situations in which we have control of the cohomology groups
of $Z$ and $M_+\sqcup M_-\cong M\setminus Z$, which allows us to determine certain Betti numbers of $M$.  
In dimension $4$, we may determine all Betti numbers, and hence the Euler characteristic
(Theorem~\ref{toric origami dim4}).  Again, this allows us to rule out a possible toric origami
structure on a specific $4$-manifold which is known to admit an effective $\T^2$ action
(Remark~\ref{non-eg: euler}).

The results in this paper were developed simultaneously to those in the recent preprint
of Ayzenberg, Masuda, Park and Zeng \cite{AMPZ}.  Their techniques rely on the assumption
that proper faces of the origami template be acyclic.  With this hypothesis, the authors are, for
the most part, able to determine the ring structure in cohomology, in terms of equivariant 
cohomology.  Our results apply to all origami templates, but our cohomological results
are only about Betti numbers.

The remainder of the paper is organized as follows.
We outline the basic notions and notation in Section~\ref{sec:background}, and compute the 
fundamental group of a toric origami
manifold in Section~\ref{sec:fund gp}. In Section~\ref{sec:coh Y/B}, we derive some auxiliary
results that we then use in Section~\ref{sec:coh results} to determine
some of the cohomology groups of toric origami manifolds.  We enumerate all of the Betti numbers
in dimension $4$.  We conclude with the full details of an example in $6$ dimensions, showing
how our techniques are tractable even in higher dimensions, when faced with specific examples.

\subsection*{Acknowledgements}
We would like to thank Allen Hatcher, Allen Knutson, Nick Sheridan and Reyer Sjamaar for useful conversations.
Remarks~\ref{rmk:type L}  and \ref{non-eg: euler} resulted from discussions of the second author with Ana Cannas da Silva,
and we are grateful for the possibility to include them here.

\section{Origami manifolds}\label{sec:background}

This is a summary of the background and set-up described in 
our previous paper \cite[\S 2]{holm-pires}, where there are more examples and details.
We include it again here to set the notation.  There is one new item: toric origami manifolds with boundary, 
which are an ingredient in Section~\ref{sec:fund gp}.

\subsection{Symplectic manifolds.}\label{se:symplectic}
We begin with a very quick review of symplectic geometry, following \cite{ca:book}.
Let $M$ be a manifold equipped with a {\bf symplectic form} 
$\omega\in \Omega^2(M)$: that is, $\omega$ is closed ($d\omega = 0$) and non-degenerate.
In particular, the non-degeneracy condition implies that $M$ must be an even-dimensional manifold.  

Suppose
that a compact connected abelian Lie group $\T= (\SS^1)^n$ acts on $M$ preserving $\omega$.
The action is {\bf weakly Hamiltonian} if for every vector $\xi\in\algt$ in the Lie algebra 
$\algt$ of $\T$, the vector field
$$
\X_\xi(p) = \frac{d}{dt}\Big[ \exp (t\xi)\cdot p \Big] \bigg|_{t=0}
$$
is a {\bf Hamiltonian vector field}. That is, we require $\omega(\X_\xi, \cdot )$ to be an exact one-form\footnote{\, The one-form $\omega(\X_\xi, \cdot )$ is automatically closed because the action preserves $\omega$.}:
\begin{equation}\label{eq:mmap}
\omega(\X_\xi, \cdot ) = d\phi^\xi.
\end{equation}
Thus each $\phi^\xi$ is a smooth function on $M$ defined by the differential equation \eqref{eq:mmap}, so determined 
up to a constant.  Taking them together, we may define a {\bf moment map}
$$
\begin{array}{rcc}
\Phi: M  & \to & \algt^* \\
 p & \mapsto & \left(\begin{array}{rcl}
                \algt & \longrightarrow & \R \\
                \xi & \mapsto & \phi^\xi(p)
                \end{array}\right).
\end{array}
$$

The action is {\bf Hamiltonian} if the moment map $\Phi$ can be chosen to be a $\T$-invariant map. Atiyah and Guillemin-Sternberg have shown that when $M$ is a compact Hamiltonian $\T$-manifold, the image $\Phi(M)$ is a convex polytope,
and is the convex hull of the images of the fixed points $\Phi(M^{\T})$ \cite{Ati82, gu-st:convexity}.

For an {\bf effective}\footnote{ \, An action is effective if no non-trivial 
subgroup acts trivially.} Hamiltonian $\T$ action on $M$, 
$
\dim(\T)\leq \frac{1}{2}\dim(M).
$
We say that the action is {\bf toric} if this inequality is in fact an equality.  
A symplectic manifold $M$ with a toric Hamiltonian $\T$ action is called a {\bf symplectic toric manifold}.  
Delzant used the moment polytope to classify symplectic toric manifolds.  

A polytope $\Delta$ in $\R^n$ is {\bf simple} if there are $n$ edges incident to each vertex, and it is {\bf rational} 
if each edge vector  has rational slope: it lies in $\Q^n\subset \R^n$.  
A simple polytope is {\bf smooth at a vertex} if the $n$ primitive vectors parallel to the edges at the vertex span the lattice 
$\Z^n\subseteq\R^n$ over $\Z$. It is {\bf smooth} if it is smooth at each vertex.  
A simple rational smooth convex polytope is called a {\bf Delzant polytope}.
We may now state Delzant's result.

\begin{theorem}[Delzant \cite{Del88}]
There is a one-to-one correspondence
$$
\left\{\begin{array}{c}
\mbox{compact toric}\\
\mbox{symplectic manifolds}\\
\end{array}\right\}
\leftrightsquigarrow
\left\{\begin{array}{c}
\mbox{Delzant polytopes}
\end{array}\right\} ,
$$
up to equivariant symplectomorphism on
the left-hand side and affine equivalence on the right-hand side.
\end{theorem}

\subsection{Origami manifolds.}
We now relax the non-degeneracy condition on $\omega$, following \cite{CGP:origami}.
A {\bf folded symplectic form} on a $2n$-dimensional manifold $M$ is a $2$-form
$\omega\in \Omega^2(M)$ that is closed ($d\omega = 0$),  
whose top power $\omega^n$ intersects the zero section
transversely on a subset $Z$
and whose restriction to points in $Z$ has maximal rank.
The transversality forces $Z$ to be a codimension $1$ embedded submanifold of $M$.  We call $Z$
the {\bf folding hypersurface} or {\bf fold}.

Let $i:Z\into M$ be the inclusion of $Z$ as a submanifold of $M$.
Our assumptions imply that $i^*\omega$ has a $1$-dimensional kernel on $Z$.
This line field is called the {\bf null foliation} on $Z$.
An {\bf origami manifold} is a folded
symplectic manifold $(M, \omega)$ whose null foliation
is fibrating: $Z\stackrel{\pi}{\to} B$ is a fiber bundle with orientable circle fibers
over a compact base $B$.
The form $\omega$ is called an {\bf origami form}
and the bundle $\pi$
is called the {\bf null fibration}. A diffeomorphism between two origami manifolds which intertwines the origami forms is called an {\bf origami-symplectomorphism}.
The definition of a Hamiltonian action only depends on $\omega$ being closed.  Thus, in the folded framework, we may
define moment maps and toric actions exactly as in Section~\ref{se:symplectic}.

An oriented origami manifold $M$ with fold $Z$ may be \textbf{unfolded} into a symplectic manifold as follows. 
Consider the closures of the connected components of $M\setminus Z$, a manifold with boundary
which consists of two copies of $Z$. 
We collapse the fibers of the null fibration by identifying the boundary points that are in the same fiber of the 
null fibration of each individual copy of $Z$. 
The result, $M_0:=(M\setminus Z) \cup B_1 \cup B_2$, is a (disconnected) smooth manifold that can be naturally endowed 
with a symplectic form which on $M_0\setminus (B_1 \cup B_2)$ coincides with the origami form on $M\setminus Z$. 
Because this can be achieved using symplectic cutting techniques, the resulting manifold $M_0$ is called the 
\textbf{symplectic cut space} (and its connected components the \textbf{symplectic cut pieces}), and the process 
is also called \textbf{cutting}. 
The symplectic cut space of a nonorientable origami manifold is the $\Z_2$-quotient 
of the symplectic cut space of its orientable double cover.

The cut space $M_0$ of an oriented origami manifold $(M,\omega)$ inherits a natural orientation. It is the orientation on $M_0$ induced from the orientation on $M$ that matches the symplectic orientation on the symplectic cut pieces corresponding to the subset of $M\setminus Z$ where $\omega^n>0$ and the opposite orientation on those pieces where $\omega^n<0$. In this way, we can associate a $+$ or $-$ sign to each of the symplectic cut pieces of an orientend origami manifold, as well as to the corresponding connected components of $M\setminus Z$.

\begin{remark}\label{rmk:orientable}
In this paper we restrict to origami manifolds whose fold is {\bf co\"orientable}: that is, the fold has an orientable neighborhood. Note that this not imply that the manifold is orientable. Indeed, for an orientable $M$, the condition that $\omega^n$ intersects the zero section transversally implies that the connected components of $M\setminus Z$ which are adjacent in $M$ have opposite signs. Since $M$ is connected, picking a sign for one connected component of $M\setminus Z$ determines the signs for all other components. As a consequence, an origami manifold  $M$ with co\"orientable fold is orientable if and only if it is possible to make such a global choice of signs for the connected components of $M\setminus Z$. 
\end{remark}

\begin{proposition}[\!\! {\cite[Props.\ 2.5 \&  2.7]{CGP:origami}}] \label{prop:model}
Let $M$ be a (possibly disconnected) symplectic manifold with a codimension two symplectic submanifold $B$ and a symplectic involution $\gamma$ of a tubular neighborhood $\mathcal{U}$ of $B$ which preserves $B$\footnote{\, In the nonco\"orientable case, the involution must satisfy  additional conditions, see \cite[Def.\ 2.23]{CGP:origami}. In the co\"orientable case, we have $B=B_1\cup B_2$ and the involution $\gamma$ maps a tubular neighborhood of $B_1$ to one of $B_2$ and vice versa.}.
Then there is an origami manifold $\widetilde{M}$ such that $M$ is the symplectic cut space of $\widetilde{M}$.  Moreover, this manifold is unique up to origami-symplectomorphism. 
\end{proposition}

This newly-created fold $Z\subset\widetilde{M}$ involves the radial projectivized normal bundle of $B\subset M$, so we call the origami manifold $\widetilde{M}$ the \textbf{radial blow-up} of  $M$ through $(\gamma,B)$. 
The cutting operation and the radial
blow-up operation  are in the following sense inverse to each other.

\begin{proposition}[\!\! {\cite[Prop.\ 2.37]{CGP:origami}}]
Let $M$ be an origami manifold with cut space $M_0$.  The radial blow-up
$\widetilde{M_0}$ is origami-symplectomorphic to $M$.
\end{proposition}

There exist Hamiltonian versions of these two operations which may be used to see that the moment map $\Phi$ for an origami manifold $M$ coincides, on each connected component of $M \setminus Z$ with the induced moment map $\Phi_i$  on the corresponding symplectic cut piece $M_i$. As a result, the moment image $\Phi(M)$ is the union of convex polytopes $\Delta_i$.

Furthermore, 
if the circle fibers of the null fibration for a connected component  $\mathcal{Z}$ of the fold $Z$ are orbits for a circle subgroup $\SS^1\subset \T$,
then  
$\Phi(\mathcal{Z})$ is a facet of each of the two polytopes corresponding to neighboring components of $M\setminus Z$.
Let us denote these  two polytopes $\Delta_1$ and $\Delta_2$.  We note that they must \textbf{agree} near $\Phi(\mathcal{Z})$: there is a neighborhood $\mathcal{V}$ of $\Phi(\mathcal{Z})$ in $\R^n$ such that $\Delta_1\cap\mathcal{V}=\Delta_2\cap\mathcal{V}$. The condition that the circle fibers are orbits 
is automatically satisfied when the action is toric, and in that case there is a classification theorem in terms of the moment data.

The moment data of a toric origami manifold can be encoded in the form of an origami template{, originally defined in~\cite[Def.\ 3.12]{CGP:origami}. Definition~\ref{def:template} below is a refinement of that original definition.  
Following~\cite[p.\ 5]{book:graph}, a \textbf{graph} $G$ consists of a nonempty set $V$ of \textbf{vertices} and a set $E$ of \textbf{edges} together with an incidence relation  that associates an edge with its two \textbf{end vertices}, which need not be distinct. Note that this allows for the existence of (distinguishable) multiple edges with the same two end vertices, and of \textbf{loops} whose two end vertices are equal.
We introduce some additional notation: let $\mathcal{D}_n$ be the set of all Delzant polytopes in $\R^n$ and $\mathcal{E}_n$ the set of all subsets of $\R^n$ which are facets of elements of $\mathcal{D}_n$.  

\begin{definition}\label{def:template}
An $n$-dimensional \textbf{origami template} consists of a graph $G$, called the \textbf{template graph}, and a pair of maps $\Psi_V: V\to\mathcal{D}_n$ and $\Psi_E:E\to\mathcal{E}_n$ such that:
\begin{enumerate}
\item if $e$ is an edge of $G$ with end vertices $u$ and $v$, then $\Psi_E(e)$ is a facet of each of the polytopes $\Psi_V(u)$ and  $\Psi_V(v)$, and these polytopes agree near $\Psi_E(e)$; and

\item if $v$ is an end vertex of each of the two distinct edges $e$ and $f$, then $\Psi_E(e)\cap\Psi_E(f)=\emptyset$.
\end{enumerate}
\end{definition}

The polytopes in the image of the map $\Psi_V$ are the Delzant polytopes of 
the symplectic cut pieces. For each edge $e$, the set $\Psi_E(e)$ is a facet 
of the polytope(s) corresponding to the end vertices of $e$. We refer to such 
a set as a {\bf fold facet}, as it is the image of the connected components of 
the folding hypersurface\footnote{ \, A nonco\"orientable  connected component 
of the folding hypersurface corresponds to a loop edge $e$.}.

With these combinatorial data in place, we may now state the classification theorem.

\begin{theorem}[\!\! {\cite[Theorem 3.13]{CGP:origami}}]\label{thm:origamiDelzant}
There is a one-to-one correspondence
$$
\left\{\begin{array}{c}
\mbox{compact toric}\\
\mbox{origami manifolds}
\end{array}\right\}
\leftrightsquigarrow
\left\{\begin{array}{c}
\mbox{origami templates}
\end{array}\right\} ,
$$
up to equivariant origami-symplectomorphism on the left-hand side,
and affine equivalence of the image of the template in $\mathbb{R}^n$ on the right-hand side.
\end{theorem}

For the purposes of this paper, we need to work with toric origami manifolds with a certain type of boundary.  
We begin by defining templates with boundary.
To do so, we now allow our graph $G$ to have {\bf dangling edges}: that is, an edge that has only
one endpoint.  Note that this is different from a loop edge.

\begin{definition}\label{def:template with boundary}
An $n$-dimensional \textbf{origami template with boundary} consists of a graph $G$, possibly including dangling edges, called the 
\textbf{template graph}, and a pair of maps $\Psi_V: V\to\mathcal{D}_n$ and $\Psi_E:E\to\mathcal{E}_n$ satisfying
the conditions (1) and (2) in Definition~\ref{def:template}.  
\end{definition}

\begin{remark}
Note that condition (1) of Definition~\ref{def:template} does not impose a constraint on a dangling edge, but condition
(2) may do so.
\end{remark}

To define the {\bf toric origami manifold with boundary} associated to an origami template with boundary, we may use a 
construction motivated by Theorem~\ref{thm:origamiDelzant}.  In this way, the boundary of the origami manifold is
contained in the fold.  More specifically, the component of the boundary corresponding to the dangling edge $e$ 
is a principal circle bundle over the toric symplectic manifold with moment image $\Psi_E(e)$.
If we collapse the circle fibers of this fibration, we obtain a toric origami manifold, possibly with boundary, with the
dangling edge $e$ removed from the template graph.

This is not the most general definition of a toric origami manifold with boundary, but it is the version that we will need
in the remainder of the paper.

The orbit space $X=M/\T $ of a toric origami manifold, possibly with boundary, 
is closely related to the origami template. When $M$ is a toric symplectic manifold, 
then the orbit space may be identified with the corresponding Delzant polytope; this identification is achieved by the moment map.
For a toric origami manifold, possibly with boundary, 
the orbit space is realized as the topological space obtained by gluing the polytopes in $\Psi_V(V)$ 
along the fold facets as specified by the map $\Psi_E$. More precisely, the orbit space is the quotient
\begin{equation}\label{eq:orbitspace}
X = \bigsqcup_{v\in V} (v,\Psi_V(v)) \Big/ \thicksim \ ,
\end{equation}
 where we identify $(u,x)\thicksim(v,y)$ if there exists an edge $e$ with endpoints $u$ and $v$ and the points $x=y\in\Psi_E(e)\subset \R^n$.
 Again, this identification is achieved by the moment map.
In simple low-dimensional examples, we can visualize the orbit space by superimposing the polytopes $\Psi_V(v)$ in 
$\mathbb{R}^n$ and indicating which of their facets to identify.
There is a deformation retraction from orbit space $X$ to the template graph.

There is a natural description of the faces of $X$. The  facets of a polytope are well-understood. The set of facets of $X$ is
$$
\bigsqcup_{\mathclap{\substack{v\in V \\ F \text{ facet of } \Psi_V(v)\\ F \text{ not a fold facet}}}} \,(v,F) \Big/ \thicksim \ ,
$$

\vskip 0.05in

\noindent where the equivalence relation is induced by the one in (\ref{eq:orbitspace}). The faces of $X$ are non-empty 
intersections of facets in $X$, together with $X$ itself.  This notion of face of the orbit space agrees with Masuda and 
Panov's definition \cite[{\S 4.1}]{masuda-panov}.

\section{The fundamental group of toric origami manifolds}\label{sec:fund gp}

We now proceed to compute the fundamental group of a toric origami manifold $M$.  As our manifolds are always
connected, we suppress the notation of a basepoint.  Key to this calculation are two lattices that arise in the 
definition of $M$ by its origami template.

\begin{definition}\label{def:NandNX}
The Delzant polytopes are subsets of $\R^n$, and the Delzant condition refers to a fixed choice of
lattice $N=\Z^n\subset \R^n$.  An important sublattice of $N$ is $N_X$, the sublattice spanned (over $\Z$)
by the normal vectors to the facets of $X=M/\T$.
\end{definition}

\noindent
Masuda and Park have investigated the relationship between the fundamental group of $M$, that
of $X$, and $N/N_X$.

\begin{proposition}[{\cite[Proposition~3.1]{masuda-park}}]\label{prop:masudapark}
Let $M$ be an orientable toric origami manifold and let $N_X$ be as in Definition~\ref{def:NandNX}.  Let $q_*\colon \pi_1(M)\to \pi_1(X)$ be the homomorphism induced from the quotient map $q\colon M\to X$.
Then there is an epimorphism
\[
\rho\colon N/N_X \times \pi_1(X) \to \pi_1(M)
\]
such that the composition $q_*\circ \rho \colon N/N_X \times \pi_1(X)\to \pi_1(X)$ is the projection on the second factor, in particular,
$\ker(\rho)$ is contained in $N/N_X$.
\end{proposition}

\begin{remark}\label{rmk:pi1 for trivial}
As Masuda and Park note, $N/N_X$ is trivial, finite cyclic or infinite cyclic.
When $N/N_X$ is trivial, then $\rho$ is an isomorphism.
\end{remark}

We aim to show that $\rho$ is an isomorphism.
We now introduce several auxiliary spaces that will allow us to identify $\pi_1(M)$.
Let $\widetilde{X}$ denote the universal cover  of the orbit space $X$.  
Let $\widetilde{M}$ be the toric origami manifold corresponding to $\widetilde{X}$.
Note that $\widetilde{M}$ is non-compact unless the original template graph is a tree.  Then there is a covering map
$V: \widetilde{M}\to M$ and an injection $V_*: \pi_1( \widetilde{M})\to \pi_1(M)$.

Choose a fundamental domain $\mathcal{D}$ for the action of the deck transformations on $\widetilde{X}$ and consider its closure $\widetilde{X}_0=\overline{\mathcal{D}}$ inside $\widetilde{X}$. This has template graph $\widetilde{G}_0$, a spanning tree of the 
original template graph $G$
for $M$, together with some extra dangling edges. More explicitly, for every edge $e$ in $G$ that is not in the spanning tree, 
there are now two dangling edges in $\widetilde{G}_0$, one emanating from each end vertex of $e$.
The manifold $\widetilde{M}_{0}$ is an origami manifold with boundary. Note that in the same way that $X$ can be recovered from $\widetilde{X}_0$ by gluing along some of the facets, and $G$ may be recovered from $\widetilde{G}_0$ by splicing the 
dangling edges described above, the origami manifold $M$ can be recovered from $\widetilde{M}_{0}$ by 
appropriately identifying boundary components to each other.

Recall that the fundamental group of $X$ is a free group $F_k$, since $X$ deformation retracts to the template graph.
The Cayley graph of the free group $F_k$ is an infinite regular tree of degree $2k$.  We may think of $\widetilde{X}$
in terms of this infinite tree, where each vertex represents a copy of $\widetilde{X}_0$ and the edges represent the facets
by which the copies of $\widetilde{X}_0$ are glued together.  We introduce auxiliary spaces $\widetilde{X}_i$, for
$i\geq 0$, which consist of the $\frac{(2k)^{i+1}-1}{2k-1}$ copies of $\widetilde{X}_0$ that are distance at
most $i$ from the ``identity" copy of $\widetilde{X}_0$ in the Cayley graph of $F_k$.  We then may define the
origami manifold with boundary $\widetilde{M}_i$ to have template with boundary $\widetilde{X}_i$.

We note that the spaces $\widetilde{X}_i$ and the spaces $\widetilde{M}_i$ are nested.  
That is, we have a commutative diagram
$$
\xymatrix{
M \ar[d] & \widetilde{M} \ar[l]_(.6){\mathrm{cover}}\ar@<-2.5ex>[d] \supseteq \cdots &  \widetilde{M}_2\ar@{}[l]|-*[@]{\supseteq} \ar[d] & \widetilde{M}_1\ar@{}[l]|-*[@]{\supseteq}\ar[d] & \widetilde{M}_0\ar@{}[l]|-*[@]{\supseteq}\ar[d] \\
X& \widetilde{X} \ar[l]_(.6){\mathrm{univ.}}^(.6){\mathrm{cover}} \supseteq \cdots &  \widetilde{X}_2\ar@{}[l]|-*[@]{\supseteq} & \widetilde{X}_1\ar@{}[l]|-*[@]{\supseteq}& \widetilde{X}_0.\ar@{}[l]|-*[@]{\supseteq}
}
$$
We also include a figure showing $X$, $\widetilde{X}_0$, $\widetilde{X}_1$ and 
$\widetilde{X}_2$.
\begin{center}
\begin{figure}[h]
\includegraphics[width=0.7in]{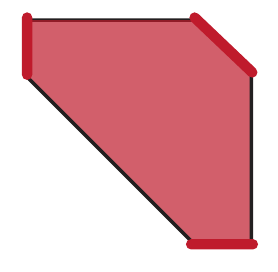} \hskip 0.2in
\includegraphics[width=0.85in]{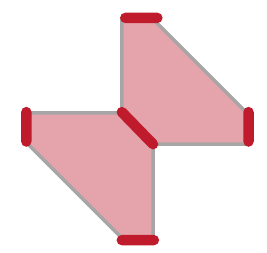} \hskip 0.2in
\includegraphics[width=1.1in]{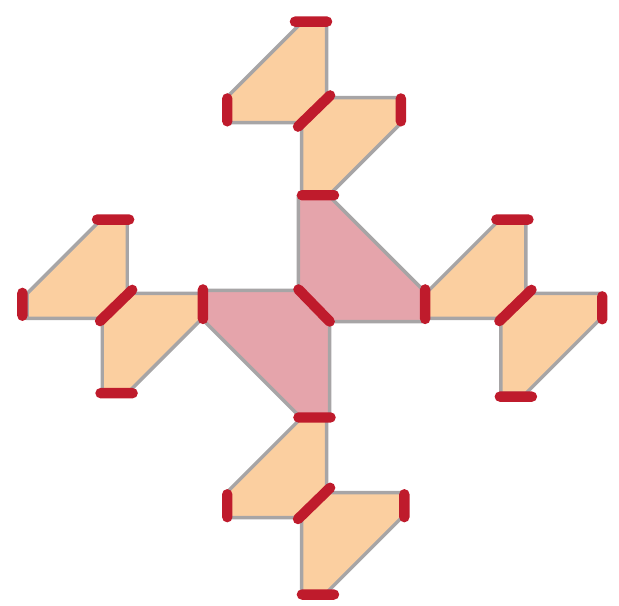} \hskip 0.2in
\includegraphics[width=1.65in]{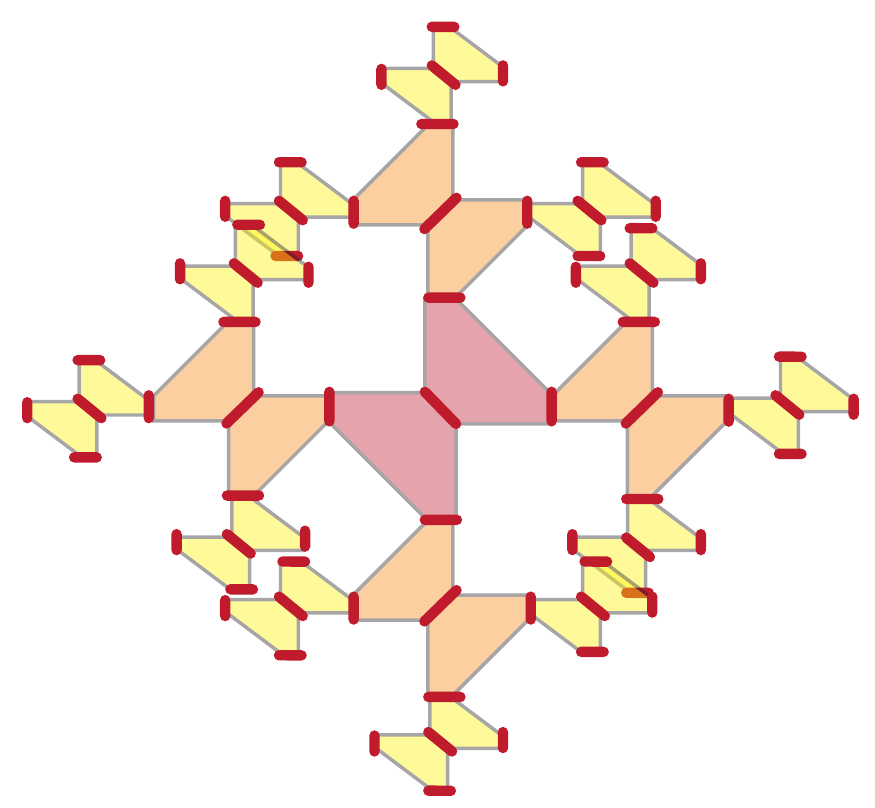} 
\caption{
From left to right: 
the moment map image of a toric origami manifold $M$ whose template graph has two vertices and three edges 
(the polytopes corresponding to the two vertices are identical and appear superimposed); and representations of 
$\widetilde{X}_0$, $\widetilde{X}_1$ and $\widetilde{X}_2$, drawn ``unfolded'' and with shrinking polytopes to 
prevent too many overlaps. The red ``boundary" facets correspond to the dangling edges of the template graph.
The moment images of each of these 3 spaces looks like the leftmost figure, but their templates are all different.}
\end{figure}
\end{center}

We now wish to compute the fundamental groups $\pi_1(\widetilde{M}_i)$ and $\pi_1(\widetilde{M})$.
Danilov computed the fundamental group of a normal toric variety associated to a fan \cite[Proposition~9.3]{danilov};
a detailed proof is given in \cite[Theorem 12.1.10]{coxlittleschenck}.

\begin{theorem}\label{thm:coxlittleschenck}
Let $\Sigma$ be a fan in $N_\mathbb{R}$ and let $N_\Sigma$ be the sublattice of $N$ generated by $|\Sigma|\cap N$. Then 
the fundamental group of the normal toric variety $X_\Sigma$ is $\pi_1(X_\Sigma)\cong N/ N_\Sigma$.
\end{theorem}

\noindent We  begin by computing $\pi_1(\widetilde{M}_i)$.

\begin{lemma}
Let $M$ be an origami manifold, possibly with boundary, such that the orbit space $X$ is simply connected
(or equivalently, the template graph $G$ of $M$ is a tree).  
Let $N_X$ be the sublattice of $N$ generated by the rays of the multi-fan corresponding to $M$.
Then the fundamental group
$\pi_1(M) \cong N/N_X$.
\end{lemma}

\begin{proof}
We proceed by induction on the number of vertices in the template graph  $G$.  
The base case is when there is a single vertex.
Then the manifold $M$ is a symplectic manifold (possibly with boundary), and the corresponding multi-fan
is in fact a fan $\Sigma$.   Then $M$ is homeomorphic the normal toric variety $X_\Sigma$, 
and the result is a direct application of Theorem~\ref{thm:coxlittleschenck}.

For the induction step, we pick a leaf vertex $v$ of $G$.  
Denote the vertex set of $G$ by $V$ and the edge set $E$.
Given the leaf vertex $v$, let $e$ be the edge that connects it to the rest of $G$ and $f_1, f_2,\ldots$ the (possibly empty) list 
of dangling edges emanating from $v$. Let $\text{Star}(v)$ be the graph with a single vertex $v$ and dangling edges 
$\tilde{e},f_1,f_2,\ldots$, where $\tilde{e}$ is the new dangling edge obtained from $e$.  Next, let $\text{Star}(V\setminus v)$
be the graph with vertex set $V\setminus \{ v\}$ and edge set $E\setminus \{ e,f_1,f_2,\dots,\}\cup \{\tilde{e}\}$, 
where $\tilde{e}$ is the new dangling edge obtained from $e$. 

We now describe a cover of $M$ with two open sets.  The first set, $A$, is a small
neighborhood in $M$ of the toric origami manifold with boundary $M_1$ whose template graph is $\text{Star}(v)$  and
orbit space is $X_1$.
We may choose $A$ so that it deformation retracts to $M_1$.  By Theorem~\ref{thm:coxlittleschenck}, $\pi_1(A)\cong N/N_{X_1}$.
The second set, $B$, is a small
neighborhood in $M$ of the toric origami manifold with boundary $M_2$ whose template graph is $\text{Star}(V\setminus v)$ and
orbit space is $X_2$.
We may choose $B$ so that it deformation retracts to $M_2$. By the induction hypothesis, $\pi_1(B)\cong N/N_{X_2}$.

We note that the intersection $A\cap B$ is the tubular neighborhood of the connected component $Z$ of the fold $\mathcal{Z}$
corresponding to the edge $e$.  It is homeomorphic to the toric variety whose fan has rays the normals to the facets in the polytope
$\Psi_V(v)$ that are adjacent to the fold facet  $\Psi_E(e)$.  Thus, we may apply Theorem~\ref{thm:coxlittleschenck} to deduce that
$\pi_1(A\cap B) \cong N/N_{A\cap B}$, where $N_{A\cap B}$ is the sublattice of $N$ spanned by the rays  described in the 
previous sentence.

We may apply the Seifert--van Kampen Theorem to deduce that
$$
\pi_1(M) \cong N/N_{X_1} {\FreeProd}_{N/N_{A\cap B}} N/N_{X_2}.
$$
As in  \cite[Proof of Theorem 12.1.10]{coxlittleschenck}, the final step is to use presentations of the groups 
$N/N_{X_1}$, $N/N_{X_2}$ and $N/N_{A\cap B}$ in terms of
generators and relations to conclude that 
$$
N/N_{X_1} {\FreeProd}_{N/N_{A\cap B}} N/N_{X_2}= N/(N_{X_1}+N_{X_2}) = N/N_X.
$$
This completes the proof.
\end{proof}

\noindent
We may now compute $\pi_1(\widetilde{M}_i)$.

\begin{cor}\label{cor:M_i}
Let $\widetilde{M}_i$ be the origami manifold with boundary with orbit space $\widetilde{X}_i$, as described
above.  For each $i\geq 0$, the fundamental group is $\pi_1(\widetilde{M}_i)\cong N/N_{\widetilde{X}_i}= N/N_X$.
\end{cor}

\begin{proof}
The only missing ingredient is to notice that $N_{\widetilde{X}_i}= N_X$ for each $i$.
\end{proof}

\noindent 
Now we  may compute $\pi_1(\widetilde{M})$.

\begin{cor}\label{cor:tilde M}
Let $\widetilde{M}$ be the toric origami manifold with boundary with orbit space $\widetilde{X}$, as described
above. The fundamental group is $\pi_1(\widetilde{M})\cong  N/N_X$.
\end{cor}

\begin{proof}
We may describe $\widetilde{M}$ as a direct limit $\widetilde{M} = \varinjlim\widetilde{M}_i$, and so we
apply  Corolloary~\ref{cor:M_i} and Exercise 2.4.11 from \cite[pp.\ 67--68]{massey} to deduce that 
$\pi_1(\widetilde{M}) \cong \varinjlim \pi_1(\widetilde{M}_i)\cong \varinjlim N/N_{X} = N/N_X$.
\end{proof}

\noindent
We next show that $N/N_X$ is a subgroup of $\pi_1(M)$.

\begin{cor}
Let $M$ be the toric origami manifold with orbit space $X$, let $\widetilde{X}$ be the universal cover of $X$, and
$\widetilde{M}$  the toric origami manifold with boundary with orbit space $\widetilde{X}$, as described
above. Then there is an injection $N/N_X\into \pi_1(M)$.
\end{cor}

\begin{proof}
As noted above, we have a covering map $V: \widetilde{M}\to M$ and therefore there is an injection 
$V_*: \pi_1( \widetilde{M})\into \pi_1(M)$.
The result now follows from Corollary~\ref{cor:tilde M}.
\end{proof}

The group $N/N_X$ must be trivial or cyclic.  When it is trivial, then $\rho$ provides an isomorphism 
$\pi_1(M)\cong \pi_1(X)$ in
Proposition~\ref{prop:masudapark}. We now tackle the two separate cases when $N/N_X$ is finite and when it is isomorphic to $\Z$.

\begin{proposition}\label{prop: pi1 for finite cyclic}
Let $M$ be the toric origami manifold with orbit space $X$.  If $N/N_X$ is a finite cyclic group, then the surjection
$\rho$ from Proposition~\ref{prop:masudapark} is an isomorphism.
\end{proposition}

\begin{proof}
We know that $\pi_1(M)\cong (N/N_X \times \pi_1(X))/\ker(\rho)$, and
that the kernel $\ker(\rho)\subset N/N_X$. Because $N/N_X$ is finite, we have an isomorphism $\pi_1(M) \cong \Z/k \Z \times F_\ell$, where 
$F_\ell\cong \pi_1(X)$ is a free group on $\ell$ generators. 
The image of $N/N_X$ under the injection $V_*$ must be in the $\Z/k\Z$ factor, since $F_\ell$ is free.  
The only way for the finite group $N/N_X$ to be a subgroup of $\Z/k \Z\cong (N/N_X)/\ker(\rho)$ is for $\ker(\rho)=\{\id\}$.
This completes the proof.
\end{proof}

\noindent Finally, we turn to the case when $N/N_X\cong \Z$.  
This situation turns out to be quite rigid.

\begin{proposition}\label{prop:YtimesT2}
The quotient $N/N_X\cong \Z$ if and only if the toric origami manifold $M^{2n}$ is 
equivariantly homeomorphic to $\T^2\times Y$, where $Y=Y_F$ is a toric symplectic manifold of dimension 
$2n-2$, the torus $\T^2$ is a toric origami manifold.
\end{proposition}

\begin{proof}
\noindent ($\Longrightarrow$)
We begin by assuming that $N/N_X\cong \Z$.  This means that $N_{\R}/(N_X)_{\R}\cong \R$, and so $U=(N_X)_{\R}$
is a hyperplane in $N_{\R}=\R^n$.  Let $u$ be a non-zero vector orthogonal to $U$.

We fix a choice of a single polytope $P=\Psi_V(v)$ in the moment image, and fix
$F$ a fold facet of $P$, that is, $F=\Psi_E(e)$ for some edge $e$ in the template.  Let $\eta$ denote the normal vector to $F$.  
Let $F_1,\dots,F_s$ be the facets of $P$ adjacent to $F$, and let $\eta_1,\dots,\eta_s$ denote the normal vectors to the facets.
Because $P$ is a Delzant polytope, $\eta,\eta_1,\dots,\eta_s$ span $N$.  
The vectors $\eta_1,\dots,\eta_s$ must span a subspace of $U$.  Combined with the fact that $P$ is Delzant, we may conclude that
$\eta_1,\dots,\eta_s$ span an $(n-1)$-dimensional subspace, so they must span all of $U$.
Each hyperplane $\eta_i^\perp$ contains $\R\cdot u$.  Thus, the affine hyperplanes $\mathcal{H}_i$ that define the facets $F_i$
all contain an affine translation of $\R\cdot u$. This means that the intersection of affine half-spaces
$$
\bigcap_{i=1}^s \mathcal{H}_i^+
$$
used to define part of $P$ can be described as an infinite prism
$$
\bigcap_{i=1}^s \mathcal{H}_i^+ = F+  \R\cdot u,
$$
where $+$ denote the Minkowski sum. 
There cannot be another non-fold facet of $P$ because the Delzant condition would force that facet to have a normal
vector pointing out of $U$, contradicting the hypothesis that $N/N_X\cong \Z$.  Therefore, there are only fold facets remaining,
and because fold facets must be isolated, there can be only one additional fold facet $\widetilde{F}$ capping off $P$.  
Note that the infinite prism $\widetilde{F}+\R\cdot u$ is identical to $F+  \R\cdot u$, and indeed $F$ and $\widetilde{F}$
have the same combinatorial type.
This description as a subset of an infinite prism is valid for each polytope in the image of $\Psi_V$.  Moreover, because
adjacent polytopes must agree near their shared fold facet, the infinite prism 
is identical in each case.
This implies that the moment image of $M$ is contained in the infinite prism  $F+\R\cdot u$. 
Moreover the template graph must be a cycle.

For our fixed choice of $P$ and $F$, we have a hyperplane $U_F = \eta^\perp\cong \R^{n-1}$, with lattice
$U_F\cap N$. Let $Y=Y_F$ denote the toric symplectic manifold of dimension $2n-2$ corresponding to the 
Delzant polytope $F\subset U_F$.  We now consider the closure $M_j$ in $M$ of a connected component 
$W_j$ of $M\setminus Z$.  This corresponds to a vertex $v_j$ in the template graph, and hence a polytope
$P_j=\Psi_V(v_j)$. We want to show that $M_j$ is equivariantly homeomorphic to 
$Y_F \times \SS^1 \times [ a_j,b_j]$.  To do so, we will think of constructing a toric symplectic manifold in the
topological manner, by taking a quotient of $P\times \T$ by an equivalence relation to get $Y_P = P\times \T/\sim$.  
In this way, if $P_j= F\times [ a_j,b_j]$, then we get a splitting  of the symplectic cut piece $C_j = Y_F \times Y_{[a_j,b_j]} 
= Y_F \times \SS^2_{[a_j,b_j]}$, and hence $M_j = Y_F \times \SS^1 \times [ a_j,b_j]$.

For the general case, we proceed as follows: let $h:F\times {[a_j,b_j]}\to P_j$ be a linear homeomorphism, preserving faces, and consider the homeomorphism $h\times\id_{\T^n}:F\times {[a_j,b_j]}\times \T^n\to P_j\times \T^n$. The closure $M_j$ of $W_j$ is obtained from $P_j\times \T^n$ by collapsing the appropriate $\SS^1\subset \T^n$ fibers over those facets of $P_j$ which are in the image $h(\partial F\times {[a_j,b_j]})$. Note that the symplectic cut piece $C_j$ corresponding to $W_j$ would be obtained by further collapsing the appropriate $\SS^1\subset \T^n$ fibers over the remaining facets of $P_j$, namely $h(F\times\{a_j\})$ and $h(F\times\{b_j\})$. In the topological construction, the circle subgroup of $\T^n$ that we collapse over a particular facet is indicated by the normal vector to that facet. Because the homeomorphism $h$ does not change the normal vectors to the facets in the image $h(\partial F\times {[a_j,b_j]})$, the map $h\times\id_{\T^n}$ induces an equivariant homeomorphism between $Y_F\times \SS^1 \times {[a_j,b_j]}$ and  $M_j$, as desired.

Thus, we have seen that the closure $M_j$ of each connected component of $M\setminus Z$ is a manifold with boundary homeomorphic to $Y_F\times \SS^1\times [a_j,b_j]$, with two boundary components that correspond to the two fold facets of the polytope $P_j$. The whole manifold $M$ is obtained from the collection of $M_j$s  by identifying them along their boundaries as prescribed by $\Psi_E$.  Thus, we may deduce that $M$ is equivariantly homeomorphic to the union of the $Y_F\times \SS^1\times [a_j,b_j]$'s  along their boundaries, and therefore is equivariantly homeomorphic to $Y_F\times \T^2$.

\medskip

\noindent
($\Longleftarrow$)
If $M$ is equivariantly homeomorphic to $\T^2\times Y$ with $Y$ toric symplectic and $\T^2$ toric origami, then its moment image is as described in the paragraphs above: see \cite[Figure~14]{CGP:origami} for the moment image of a toric origami $\T^2$. Thus we must have that $N/N_X\cong \Z$.
\end{proof}

\begin{definition}\label{def:prismatic origami}
In the case when the quotient $N/N_X\cong \Z$ and $M\cong \T^2\times Y$,
we call the toric origami manifold {\bf prismatic}.
\end{definition}

\begin{cor}\label{cor:pi1 for Z}
If a toric origami manifold $M$ is prismatic, then its fundamental group is 
$\pi_1(M)=\mathbb{Z}^2$.
\end{cor}

\begin{proof}
The fundamental group is a homeomorphism invariant, and $\pi_1(\T^2\times Y)=\pi_1(\T^2)\times \pi_1(Y)$, where $\pi_1(\T^2)=\mathbb{Z}^2$ and $Y$ is simply connected because it is a toric symplectic manifold.
\end{proof}

\noindent 
We now have all the necessary ingredients to compute $\pi_1({M})$.

\begin{theorem}\label{thm:fund group}
Let $M$ be an orientable toric origami manifold with orbit space $X$, and let $N$ and $N_X$ be as in Definition~\ref{def:NandNX}. Then the fundamental group of $M$ is 
$$\pi_1(M)\cong N/N_X\times\pi_1(X).$$
\end{theorem}

\begin{proof}
By Remark~\ref{rmk:pi1 for trivial}, Proposition~\ref{prop: pi1 for finite cyclic} and Corollary~\ref{cor:pi1 for Z}, the only thing missing is to check that for $M$ prismatic, the fundamental group of the orbit space is $\pi_1(X)=\Z$. This is true because $X$ deformation retracts onto the template graph, which as remarked in the proof of Proposition~\ref{prop:YtimesT2} is a cycle.
\end{proof}

\noindent 
In particular, this allows us to deduce that Masuda and Park's map $\rho$ \cite{masuda-park} is an isomorphism.

\begin{cor}
The epimorphism $\rho\colon N/N_X \times \pi_1(X) \to \pi_1(M)$ from Proposition~\ref{prop:masudapark} is an isomorphism.
\end{cor}

\begin{proof}
By Remark~\ref{rmk:pi1 for trivial}, Proposition~\ref{prop: pi1 for finite cyclic}, we are only left with checking that $\rho$ is an isormorphism when $M$ is prismatic. Recall that in that case $N/N_X=\Z$ and $\pi_1(X)=\Z$. 

We know that $\pi_1(M)\cong (N/N_X \times \pi_1(X))/\ker(\rho)\cong \Z^2/\ker(\rho)$, and that $\ker(\rho)\subset N/N_X\cong\Z$. Since $\pi_1(M)=\Z^2$, the only possibility is that the kernel is the trivial subgroup $\ker(\rho)=\{\id\}$.
\end{proof}

Another consequence of Theorem~\ref{thm:fund group} is a characterization of simply connected toric origami manifolds.  
The following result indicates that the key assumption in our previous work \cite{holm-pires}, namely that the origami template
be acyclic, is a natural topological hypothesis.

\begin{cor}\label{cor:simply connected}
A toric origami manifold is simply connected if and only if its origami template is acyclic.
\end{cor}

\begin{proof}
If $M$ is a toric origami manifold and has at least one cycle in its template graph, then
there must be at least one infinite cyclic factor in $\pi_1(M)$ and $M$ not simply connected.

If the template graph is acyclic, then the $\pi_1(X)$ factor of $\pi_1(M)$ is trivial. In addition, any polytope corresponding to a leaf of the template graph has at least one vertex not contained in a fold facet. By the Delzant condition at that vertex, the lattice quotient $N/N_X$ is trivial. Thus $\pi_1(M)=\{\id\}$ and $M$ is simply connected.
\end{proof}

\noindent 
In Table~\ref{table:hirz} below, we show examples where $N/N_X$ takes on all possible types of group.

\renewcommand{\arraystretch}{1.3}
\begin{table}[h]
\begin{tabular}{ c | m{3.75cm} | m{5cm} | m{5.25cm}}\label{ref:hirz}
$M$  &  \centering $\SS^2\times\SS^1\times\SS^1 \cong \SS^2\times\T^2$ & \centering $\SS^3\times \SS^1$ & \centering $L(k ;1)\times\SS^1$ \tabularnewline \hline
\centering {$\Phi(M)$} & 
\centering \includegraphics[height=0.6in]{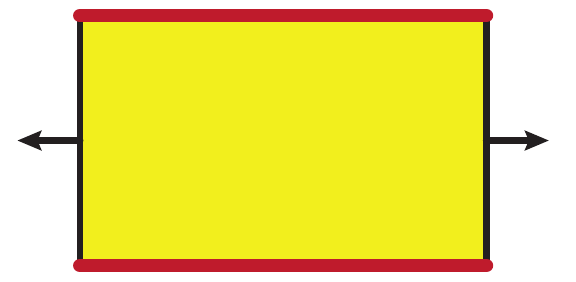} &
\centering \includegraphics[height=0.6in]{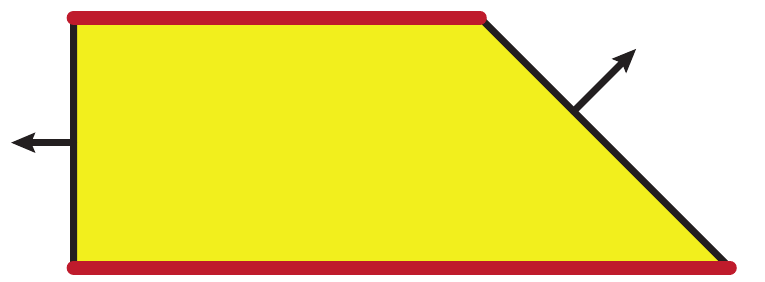} &
\centering \includegraphics[height=0.6in]{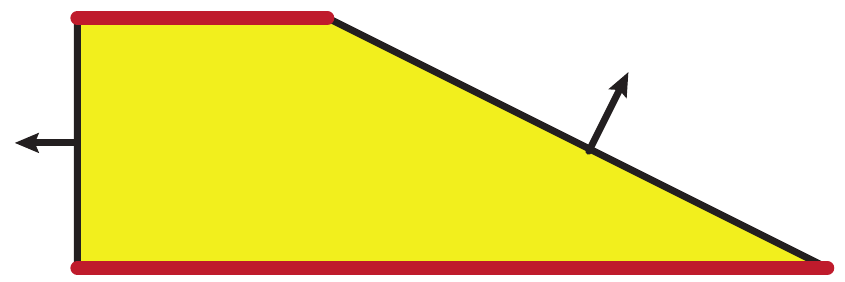}\tabularnewline \hline
$N/N_X$ & \centering $\Z$ & \centering $\{\id\}$ & \centering $\Z/k\Z$ \tabularnewline \hline
$\pi_1(M)$  &  \centering $\Z\times\Z$ & \centering $\Z$ & \centering $\Z/k\Z \times \Z$ \tabularnewline \hline
\end{tabular}\vskip 0.1in
\caption{Examples of the possible types of $N/N_X$.  In each case, the template graph has two vertices, connected
to one another by two edges.  Each quotient space has two facets, with facet normals indicated in the figures.}\label{table:hirz}
\end{table}
\renewcommand{\arraystretch}{1.2}

\begin{remark}\label{rmk:type L}
The form of the fundamental group of a toric origami manifold given by Theorem \ref{thm:fund group} 
excludes certain manifolds from admitting such a structure.
For example, a non-trivial finite cyclic group $\Z/k\Z$ cannot occur as the fundamental group of a toric 
origami manifold (as noted in the Proof of \cite[Corollary 3.2]{masuda-park} for the non-simply connected case). 
To verify this, we note that if $M$ is a toric origami manifold and has at least one cycle in its template graph, then
there must be at least one infinite cyclic factor in $\pi_1(M)$.  On the other hand, if the template graph is acyclic,
then $M$ is simply connected by Corollary~\ref{cor:simply connected}.
Orlik and Raymond introduced manifolds so-called  ``of type $L$'' as some of the building blocks for 4-manifolds admitting toric actions  \cite{orlikraymond}. More precisely, \cite[Theorem VI.1]{pao} states that every orientable compact smooth 4-manifold that admits an effective smooth action of $\T^2$ with at least one fixed point is diffeomorphic to a connected sum of copies of $\SS^4$, $\C P^2$ , $\overline{\C P^2}$, $\SS^2\times\SS^2$, $\SS^1\times\SS^3$, $L_n$ and $L_n'$, for $n\geq 2$. The fundamental group of the manifolds of type $L$ is $\pi_1(L_n)=\pi_1(L_n')=\Z/n\Z$ (see \cite[p.\ 296]{pao}), which implies that they do not admit a toric origami structure. It is easy to see that all the other building blocks do admit toric origami structures.
\end{remark}

\section{The cohomology of $M\setminus Z$}\label{sec:coh Y/B}

In this section we obtain results about the cohomology of open toric symplectic manifolds of the form $Y\setminus \mathcal{B}$, 
where $Y$ is a compact toric symplectic manifold and $\mathcal{B}$ is a (not necessarily connected) codimension two 
toric symplectic submanifold of $Y$. This is exactly the form that the connected components of $M\setminus Z$ take. 
In the following section, we will assemble these pieces in a Mayer-Vietoris sequence and deduce facts about the
cohomology of $M$. 

In this and the following section, we write
$$
b_i(X)=\rank \left(H_i^{\phantom{k}}(X;\Z)\right) \mbox{ and } b^i(X)=\rank \left(H^i(X;\Z)\right)
$$ 
to denote the $i^{\mathrm{th}}$ (respectively, homology and cohomology) Betti numbers of the space $X$.

We begin by stating a result about the Euler characteristic of a manifold $Y\setminus \mathcal{B}$. This fact is known in greater generality, see for example \cite[\S 4.5]{fulton}.

\begin{proposition}\label{co:euler}
The Euler characteristics of $Y$ and $\mathcal{B}$ are additive:
$$
\chi(Y\setminus\mathcal{B}) = \chi(Y)-\chi(\mathcal{B}).
$$
\end{proposition}

\begin{proof}
We consider the long exact sequence for the pair $(Y,Y\setminus\mathcal{B})$,
with integer coefficients understood:
\begin{equation}\label{LES:euler}
\xymatrix{
\cdots\ar[r]  & H_{*}(Y\setminus\mathcal{B})\ar[r] & H_{*}(Y)\ar[r] &H_{*}(Y,Y\setminus\mathcal{B})\ar[r]
&  H_{*-1}(Y\setminus\mathcal{B})\ar[r] & \cdots.
}
\end{equation}
By \cite[Proposition 3.46]{hatcher:AT}, noting that $\mathcal{B}$ is compact and locally contractible and $Y$ is an orientable manifold, and by Poincar\'e duality for the manifold $\mathcal{B}$, we can replace the relative terms:
\begin{equation}\label{eq:hatcher and poincare duality}
H_{*}(Y,Y\setminus\mathcal{B})\cong H^{2n-*}(\mathcal{B})\cong H_{*-2}(\mathcal{B}).
\end{equation}

For simplicity of bookkeeping, we write $b_{-2}(\mathcal{B})=b_{-1}(\mathcal{B})=0$, since these ranks correspond, via the Poincar\'e duality in (\ref{eq:hatcher and poincare duality}), to
$$H_{-2}(\mathcal{B})\cong H^{2n}(\mathcal{B}) =0\mbox{ and } H_{-1}(\mathcal{B})\cong H^{2n-1}(\mathcal{B}) =0.$$
Now taking the alternating sum of the ranks of the terms in the sequence \eqref{LES:euler}, we obtain:
\begin{align*}
0&=\sum_{k=0}^{2n}(-1)^k \left[b_k(Y\setminus\mathcal{B})-b_k(Y)+b_{k-2}(\mathcal{B})\right]\\
&= \sum_{k=0}^{2n}(-1)^k b_k(Y\setminus\mathcal{B}) - \sum_{k=0}^{2n}(-1)^k b_k(Y)+\sum_{j=0}^{2n-2}(-1)^j b_{j}(\mathcal{B})\\
&=\chi(Y\setminus\mathcal{B}) -\chi(Y)+\chi(\mathcal{B}).
\end{align*}
This completes the proof.
\end{proof}

We now prove a Lemma related to \cite[Proposition 3.46]{hatcher:AT} that is a dual version of what is commonly called Alexander-Lefschetz duality.  We adapt the very explicit proof 
given by M\o ller \cite[Theorem~4.92]{moller} to this dual version, 
taking into account our special case that $\mathcal{B}$
is a submanifold of $Y$, suitably oriented.

\begin{lemma}\label{le:coAlexander}
There is an isomorphism $H^{j}(Y,Y\setminus\mathcal{B};\Z) \cong H_{2n-j}(\mathcal{B};\Z)$.
\end{lemma}

\begin{proof}
Let $U$ be an open neighborhood of $\mathcal{B}$ in $Y$.
We begin by recalling the particulars of cap products.  Both $Y$ and $\mathcal{B}$ are 
$\Z$-orientable manifolds,
and so we have the following maps induced by taking a cap product with appropriate orientation
classes.
\begin{enumerate}
\item[\ding{172}] For the compact manifold $Y$, we have orientation class $\mu_Y\in H_{2n}(Y;\Z)$,
which gives
$$
\xymatrix{
H^j(Y;\Z)\ar[rr]^(0.45){\mu_Y\cap -} & & H_{2n-j}(Y;\Z).
}
$$

\item[\ding{173}] For the manifold with boundary $Y\setminus U$, we have the relative orientation class 
$$
\mu_{Y\setminus U}\in H_{2n}(Y\setminus \mathcal{B}, U\setminus  \mathcal{B};\Z),
$$
which gives
$$
\xymatrix{
H^j(Y\setminus \mathcal{B};\Z)\ar[rr]^(0.4){\mu_{Y\setminus U}\cap -} &  & H_{2n-j}(Y\setminus \mathcal{B}, U\setminus  \mathcal{B};\Z).
}
$$

\item[\ding{174}] For the compact manifold $\mathcal{B}$, we have the relative orientation class 
$$
\mu_{\mathcal{B}}\in H_{2n}(U, U\setminus  \mathcal{B};\Z).
$$
By pre-composing with the excision isomorphism, we have
$$
\xymatrix{
H^j(Y,Y\setminus\mathcal{B};\Z) \cong H^j(U,U\setminus\mathcal{B};\Z) \ar[rr]^(0.65){\mu_{\mathcal{B}}\cap -} 
& & H_{2n-j}(U;\Z).
}
$$
\end{enumerate}

We use these maps to produce a diagram, with integer coefficients,
$$\small{
\xymatrix{
\cdots\ar[r] & H^{j-1}(Y\setminus\mathcal{B})\ar[r]\ar[d]_{\mbox{\ding{173}}} & H^j(Y,Y\setminus\mathcal{B})\ar[r]\ar[d]_{\mbox{\ding{174}}} & H^j(Y)\ar[r]\ar[d]_{\mbox{\ding{172}}} & H^j(Y\setminus\mathcal{B})\ar[r]\ar[d]_{\mbox{\ding{173}}} & \cdots \\
\cdots\ar[r] & H_{2n-(j-1)}(Y\setminus\mathcal{B}, U\setminus\mathcal{B})\ar[r] & H_{2n-j}(U)\ar[r] & H_{2n-j}(Y)\ar[r] & H_{2n-j}(Y\setminus\mathcal{B}, U\setminus\mathcal{B})\ar[r] & \cdots 
}}
$$
where the top row is the long exact sequence of the pair $(Y,Y\setminus\mathcal{B})$, and the bottom
row is the long exact sequence of the pair $(Y,U)$ where the terms $H_k(Y,U)$ are replaced by
$H_k(Y\setminus\mathcal{B}, U\setminus\mathcal{B})$ via excision.
This diagram commutes: the ``up to sign" discrepancy in \cite[Proof of Theorem~4.92]{moller} 
disappears because we may
choose $Y$ and $\mathcal{B}$ to be compatibly oriented.

We next take a limit over the poset $\mathcal{U}$ of neighborhoods $U$ containing $\mathcal{B}$
to obtain a limit diagram
\begin{equation}\label{eq:coAlexander}
\begin{array}{c}
\xymatrix{
\cdots\ar[r] & H^{j-1}(Y\setminus\mathcal{B})\ar[r]\ar[d]_{\mbox{\ding{183}}} & H^j(Y,Y\setminus\mathcal{B})\ar[r]\ar[d]_{\mbox{\ding{184}}} & H^j(Y)\ar[r]\ar[d]_{\mbox{\ding{182}}} & H^j(Y\setminus\mathcal{B})\ar[r]\ar[d]_{\mbox{\ding{183}}} & \cdots \\
\cdots\ar[r] & H_{2n-(j-1)}(Y\setminus\mathcal{B})\ar[r] & H_{2n-j}(\mathcal{B})\ar[r]_{\mbox{\ding{74}}} & H_{2n-j}(Y)\ar[r] & H_{2n-j}(Y\setminus\mathcal{B})\ar[r] & \cdots 
}
\end{array}
\end{equation}
which still commutes, and the bottom sequence remains exact under the limit.  We observe that the
map \ding{74} is induced by inclusion. We also note that
\ding{182} and \ding{183} are Poincar\'e Duality isomorphisms.  We now apply the Five Lemma to
deduce that \ding{184} is an isomorphism, completing the proof.
\end{proof}

We return to the long exact sequence of the pair $(Y,Y\setminus\mathcal{B})$, which is the top row in the diagram \eqref{eq:coAlexander}, with integer coefficients understood. The toric symplectic manifold $Y$ has cohomology concentrated in even degrees, up to 
degree $2n$. The space $\mathcal{B}$ is a disjoint union of toric symplectic manifolds, therefore its homology is concentrated in even degrees up to degree $2n-2$. 
Then by Lemma~\ref{le:coAlexander} the long exact
sequence splits into $4$-term exact sequences, with integer coefficients,
\begin{equation}\label{4ESpairYB}
0\to H^{2k-1}(Y\setminus\mathcal{B}) \to H^{2k}(Y,Y\setminus\mathcal{B}) \stackrel{\varphi_k}{\to} H^{2k}(Y) \to
H^{2k}(Y\setminus\mathcal{B}) \to 0.
\end{equation}
Thus, we may always identify $H^{2k-1}(Y\setminus\mathcal{B}) \cong \ker(\varphi_k)$ and 
$H^{2k}(Y\setminus\mathcal{B})\cong \mathrm{coker}(\varphi_k)$.

Let us now look more carefully at the map $\varphi_k$.  We have a
diagram
\begin{equation}\label{eq:rectangle}
\begin{array}{c}
\xymatrix{
 H^{2k}(Y,Y\setminus\mathcal{B})\ar[r]^(.6){\varphi_k}\ar[d]_{\mbox{\ding{184}}} & H^{2k}(Y)\ar[d]^{\mbox{\ding{182}}}  \\
 H_{2n-2k}(\mathcal{B})\ar[r]_{\mbox{\ding{74}}}\ar[d]_{\cong}^{PD} & H_{2n-2k}(Y)\ar[d]_{\cong}^{PD}\\
 H^{2k-2}(\mathcal{B})\ar@{-->}[r]_{\widetilde{\varphi}_k} & H^{2k}(Y)
}
\end{array},
\end{equation}
where all vertical maps are isomorphisms, and we define $\widetilde{\varphi}_k$ to be the map that makes the 
bottom square commute.
Recall from the comments after \eqref{eq:coAlexander} that \ding{74} is the natural map induced by the 
inclusion $i:\mathcal{B}\into Y$.  The homology groups of $Y$ and $\mathcal{B}$ are isomorphic to the 
Chow homology groups of those varieties.  The Chow groups of smooth toric varieties are very explicitly 
understood: they are spanned by classes, one for each $\T$-invariant subvariety.  A subvariety in $\mathcal{B}$
may be regarded as a subvariety of $Y$, and so the map \ding{74} maps the corresponding class on $\mathcal{B}$
to the class on $Y$.

When we apply Poincar\'e duality, we have very explicit presentations of the cohomology rings
$H^*(Y;\Z)$ and $H^*(\mathcal{B};\Z)$ as the face rings of the corresponding polytopes, modulo linear
relations.  That is, when the moment polytope $\Delta_Y$ for $Y$ has facets $F_1,\dots, F_d$, we may describe
\begin{equation}\label{eq:stanley-reisner}
H^*(Y;\Z) \cong \frac{\Z[y_1,\dots,y_d]}{\left\langle \begin{array}{c} \prod_{i\in I} y_i\ \Big|\ \bigcap_{i\in I} F_i = \emptyset \\ 
+\mbox{ linear terms}\end{array}
 \right\rangle},
\end{equation}
where each $y_i$ has degree $2$, and is the Poincar\'e dual of the codimension $2$ toric symplectic submanifold corresponding
to the facet $F_i$. The linear terms are determined by the geometry of the normal 
vectors to the facets.  We note, for bookkeeping purposes, that there are precisely $n=\frac{1}{2}\dim(Y)$ 
independent linear relations. That is, $\rank(H^2(Y)) = d-n$.
Equation~\eqref{eq:stanley-reisner} is the content of the Danilov-Jurkiewicz Theorem, which is carefully described
in \cite[Theorem~12.4.4]{coxlittleschenck}.  

For a connected component $B_s\subset \mathcal{B}$,
the moment image $\Delta_{B_s}$ of $B_s$ is one of the facets $F_s$.  The facets of $\Delta_{B_s}$
are each an intersection $F_s\cap F_j$, and so as above, we may describe
$$
H^*(B_s;\Z) \cong \frac{\Z[b_{j_1},\dots,b_{j_m}]}{\left\langle \begin{array}{c} \prod_{i\in I} b_{j_i}\ \Big|\ \left(\bigcap_{i\in I} F_{j_i}\right)\cap F_s = \emptyset \\ 
+\mbox{ linear terms}\end{array}
 \right\rangle}.
$$
Because the $y_i$ and $b_i$ are Poincar\'e duals to explicit submanifolds of $Y$ and $B_s$ respectively, and because \ding{74}
is induced by inclusion,
we may derive an explicit formula for $\widetilde{\varphi}_k$.
For the component $B_s\subset \mathcal{B}$ and a single
monomial $ \prod_{i\in I} b_{j_i}\in H^{2k-2}(B_s;\Z)$,
$$
\widetilde{\varphi}_k \left( \prod_{i\in I} b_{j_i} \right) = y_s\cdot \prod_{i\in I} y_{j_i}.
$$
This is not a ring map, as expected.

The following definition extends the notion of prismatic origami manifold 
(Definition~\ref{def:prismatic origami}) to a wider context.  Let $A$ be an open toric symplectic manifold 
with open moment polytope $\Delta_A$.  The lattice $N_{\Delta_A}$ is the sublattice of $N$ spanned
by the normal vectors to the facets of $\Delta_A$.

\begin{definition}
An open toric symplectic manifold $A$ with moment polytope $\Delta_A$ is  
{\bf prismatic}  if the quotient of lattices $N/N_{\Delta_A}$ is $\Z$.
\end{definition}

We now turn to $\widetilde{\varphi}_1:H^0(\mathcal{B};\Z)\to H^2(Y;\Z)$. The group 
$H^0(\mathcal{B};\Z)\cong\Z^r$ has one generator for each connected component of 
$\mathcal{B}$, each corresponding to a facet in $\Delta_Y$. The group $H^2(Y;\Z)$ 
has one generator for each facet of $\Delta_Y$, modulo linear relations. 
By our explicit description above, the map 
$\widetilde{\varphi}_1$ takes the generator of $H^0(\mathcal{B})$ corresponding 
to a facet $F_s$ of the polytope $\Delta_Y$ to the generator  $y_s\in H^2(Y;\Z)$ 
corresponding to the same facet.   
We may use our explicit description of $\widetilde{\varphi}_1$ to determine 
$\ker(\widetilde{\varphi}_1)\cong H^1(Y\setminus \mathcal{B};\Z)$  in general.

\begin{lemma}\label{le:kernel of map prismatic vs not prismatic}
The kernel of the map $\widetilde{\varphi_1}:H^0(\mathcal{B};\Z)\to H^2(Y;\Z)$ is $\ker(\widetilde{\varphi}_1)\cong\mathbb{Z}$ if $Y\setminus\mathcal{B}$ is prismatic and trivial otherwise.
\end{lemma}

\begin{proof}
Without loss of generality, we may assume that $\mathcal{B}$ corresponds to the disjoint union of facets 
$F_1,\ldots,F_r$ of $\Delta_Y$. Let $u_1,\ldots,u_d$ be the primitive outward pointing normals to all the facets
of $\Delta_Y$. Then 
$$H^0(\mathcal{B};\Z)=\mathbb{Z}x_1\oplus\ldots\oplus\mathbb{Z}x_r \,\,\,\,\,\,\,\,\,\,\mbox{  and  }\,\,\,\,\,\,\,\,\,\, H^2(Y;\Z)=\mathbb{Z}y_1\oplus\ldots\oplus\mathbb{Z}y_d /J,$$
 where $J$ is the ideal of linear relations, which are $\sum_{i=1}^d\langle v,u_i\rangle y_i$, for all $v\in N$. 
Henceforth, we will abuse notation, and let $y_i$ denote the equivalence class $y_i+J\in H^2(Y;\Z)$.

The map $\widetilde{\varphi_1}$ is given by $\widetilde{\varphi_1}(x_i)=y_i$.
Therefore an element $\sum_{i=1}^r a_i x_i\in H^0(\mathcal{B};\Z)$ is in the kernel of $\widetilde{\varphi}_1$ 
if and only if there exists a $v\in N$ such that 
$$\sum_{i=1}^r a_i y_i=\sum_{i=1}^d \langle v,u_i\rangle  y_i.$$
Equivalently, $v\in N$ must satisfy
\begin{equation}\label{eq:conditions on v} 
\langle v,u_i\rangle =a_i \mbox{ for } i=1,\ldots,r \,\,\,\mbox{ and }\,\,\, \langle v,u_i\rangle =0 \mbox{ for } i=r+1,\ldots,d .
\end{equation}
The second half of \eqref{eq:conditions on v} means that $v\in(N_{\Delta_{Y\setminus\mathcal{B}}})^\perp$.

If $Y\setminus\mathcal{B}$ is not prismatic, then the lattice quotient $N/N_{\Delta_{Y\setminus\mathcal{B}}}$ is either trivial or finite cyclic, which implies that $(N_{\Delta_{Y\setminus\mathcal{B}}})^\perp$ is trivial and so $\ker(\widetilde{\varphi}_1)$ is trivial as well.

If $Y\setminus\mathcal{B}$ is prismatic, then $(N_{\Delta_{Y\setminus\mathcal{B}}})^\perp\cong\mathbb{Z}$ and a generator $v$ of $(N_{\Delta_{Y\setminus\mathcal{B}}})^\perp$ will provide a non-zero element $\sum_{i=1}^d \langle v,u_i\rangle  y_i\in J$, and therefore gives a  generator $\sum_{i=1}^r \langle v,u_i\rangle  x_i$ of  $\ker(\widetilde{\varphi}_1)$.
\end{proof}

\noindent We can now prove the main result of the section.

\begin{theorem}\label{prop:prismatic vs not prismatic}
Let $Y$ be a $2n$-dimensional toric symplectic manifold with moment polytope $\Delta_Y$, and
let $d$ denote the number of facets of $\Delta_Y$.
Let $\mathcal{B}$ be a non-empty
codimension $2$ toric symplectic submanifold with $r$ connected components.  We may compute
the following Betti numbers of $Y\setminus \mathcal{B}$.

\begin{center}
  \begin{tabular}{ r | c | c | }
    & $Y\setminus\mathcal{B}$ prismatic & $Y\setminus\mathcal{B}$ not prismatic\\ \hline    
    $b^0(Y\setminus\mathcal{B})$ & $1$ & $1$ \\ \hline
    $b^1(Y\setminus\mathcal{B})$ & $1$ & $0$ \\ \hline
    $b^2(Y\setminus\mathcal{B})$ & $d-n-1$ & $d-n-r$ \\ \hline
    $b^{2n-1}(Y\setminus\mathcal{B})$ & $1$ & $r-1$ \\ \hline
    $b^{2n}(Y\setminus\mathcal{B})$ & $0$ & $0$ \\ 

    \hline
  \end{tabular}
\end{center}
\end{theorem}

\begin{proof}
We begin by noting that if $Y\setminus \mathcal{B}$ is prismatic then $\mathcal{B}$ necessarily has 
$r=2$ connected components. Using this,
we will determine all rows simultaneously in the prismatic and non-prismatic cases.

\vskip 0.1in
\noindent \fbox{{\bf Computation of} $b^0(Y\setminus\mathcal{B})$.} The manifold $Y\setminus{B}$ is connected,
so $b^0(Y\setminus\mathcal{B})=1$. \hfill \ding{52}

\vskip 0.1in
\noindent \fbox{{\bf Computation of} $b^1(Y\setminus\mathcal{B})$.} Recall that 
$\ker(\widetilde{\varphi}_1)\cong H^1(Y\setminus \mathcal{B};\Z)$.
Lemma \ref{le:kernel of map prismatic vs not prismatic}  says $b^1(Y\setminus\mathcal{B})=1$ when $Y\setminus\mathcal{B}$
is prismatic, and $b^1(Y\setminus\mathcal{B})=0$ otherwise. \hfill \ding{52}

\vskip 0.1in
\noindent \fbox{{\bf Computation of} $b^2(Y\setminus\mathcal{B})$.} 
We now identify the terms of \eqref{4ESpairYB} in the case $k=1$:
$$\xymatrix{
0\ar[r] & H^{1}(Y\setminus\mathcal{B}) \ar[r]& H^{2}(Y,Y\setminus\mathcal{B})\cong H^{2}(\mathcal{B})\cong \Z^r  \ar[r] 
 & H^{2}(Y)\cong \Z^{d-n}\ar[r] &
H^{2}(Y\setminus\mathcal{B}) \ar[r] & 0
}.
$$
A dimension count  proves that
$$b^{2}(Y\setminus\mathcal{B})=b^{1}(Y\setminus\mathcal{B})+d-n-r.$$
Substituting $b^1(Y\setminus\mathcal{B})$, and $r=2$ in the prismatic case, completes the calculation.
\hfill \ding{52}

\vskip 0.1in
\noindent \fbox{{\bf Computation of} $b^{2n}(Y\setminus\mathcal{B})$.} 
We note that $Y\setminus\mathcal{B}$ is homotopy equivalent to a manifold $X$ with boundary $\mathcal{Z}$.
Poincar\'e duality for manifolds with boundary implies that
$$
H^{2n}(Y\setminus\mathcal{B};\Z)\cong H^{2n}(X;\Z) \cong H_0(X,\mathcal{Z};\Z).
$$
We note that relative cohomology $H_0(X,\mathcal{Z};\Z)\cong
\widetilde{H}_0((X/\mathcal{Z};\Z)$, and this is $0$ because
$X/\mathcal{Z}$ is connected.
\hfill \ding{52}

\vskip 0.1in
\noindent \fbox{{\bf Computation of} $b^{2n-1}(Y\setminus\mathcal{B})$.} 
We now identify the terms of \eqref{4ESpairYB} in the case $k=n$:
$$
0\to H^{2n-1}(Y\setminus\mathcal{B}) \to H^{2n}(Y,Y\setminus\mathcal{B})\cong H^{2n-2}(\mathcal{B})\cong \Z^r \to H^{2n}(Y) \cong \Z \to
H^{2n}(Y\setminus\mathcal{B})=0\to 0
$$
We have the rightmost  equality $H^{2n}(Y\setminus\mathcal{B})=0$ by the previous computation.  
A dimension count now proves that $b^{2n-1}(Y\setminus\mathcal{B})=r-1$.
Substituting $r=2$ in the prismatic case completes the calculation.
\hfill\hfill\hfill\hfill\hfill\hfill\hfill\hfill\hfill\hfill\hfill\hfill\hfill\hfill\hfill\hfill\hfill \ding{52}
\end{proof}

We note that when $n=2$, Theorem~\ref{prop:prismatic vs not prismatic} gives all the Betti numbers
of $Y\setminus\mathcal{B}$.  Even in higher dimensions, in specific examples, it is often tractable to compute the various maps
$\widetilde{\varphi}_k$, and to compute all of the Betti numbers, as well as torsion in the cohomology groups
$H^*(Y\setminus\mathcal{B};\Z)$.  We conclude the section with such an example.

\begin{example}\label{ex:truncated cube}
Let $Y$ be the toric variety $\C P^1\times\C P^1\times\C P^1$ blown up at one fixed point.  This has the moment polytope
shown in Figure~\ref{fig:truncated cube}.
\begin{figure}[ht]
\centering
{
\includegraphics[scale=0.5]{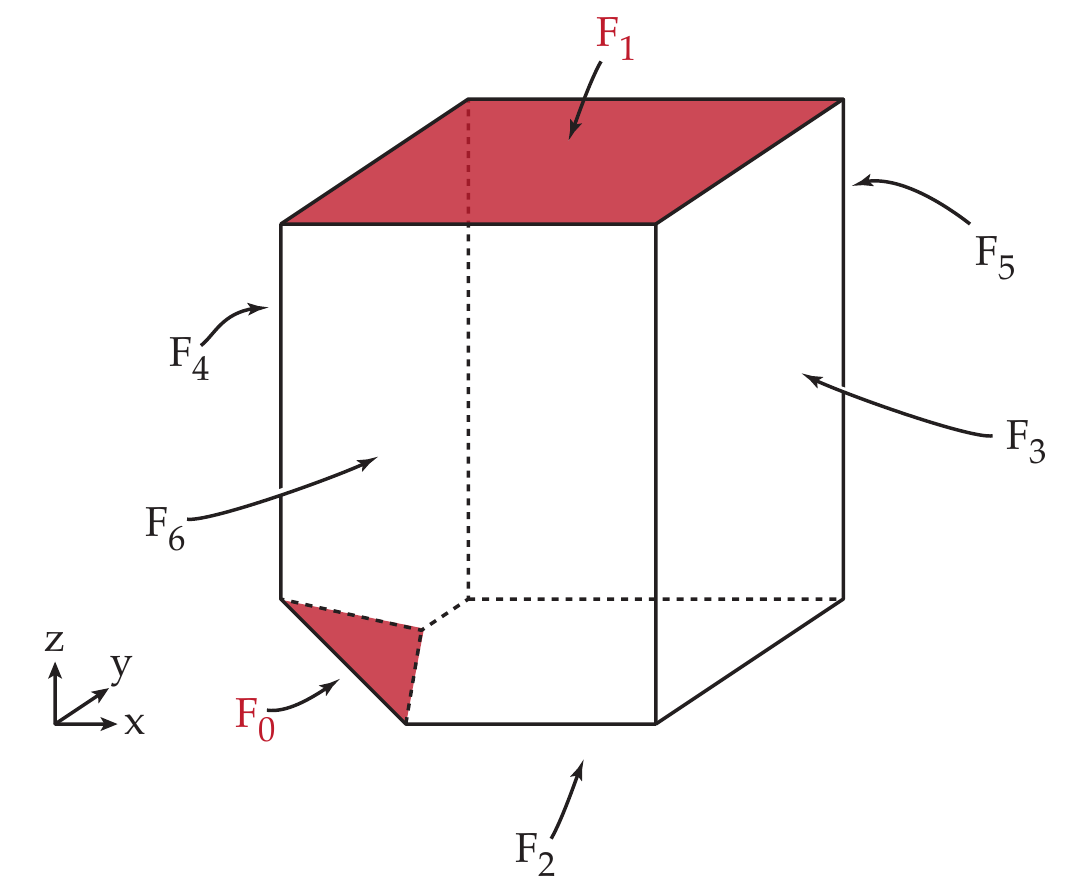} 
}
\caption{The moment map image for the $\T^3$ action on 
$\C P^1\times\C P^1\times\C P^1$ blown up at one fixed point.  The polytope is a cube truncated at one vertex.
}
\label{fig:truncated cube}
\end{figure}
In our calculations below, we will use the linear relations to simplify the presentations of the cohomology rings: that is,
we will use them to reduce the number of degree $2$ generators.  We have:
\begin{eqnarray*}
H^*(Y;\Z) & = &  \frac{\Z[y_0,y_1,y_2,y_3,y_4,y_5,y_6]}{\left\langle 
\begin{array}{c}
y_1y_2\ ,\ y_3y_4\ , \ y_5y_6\ , \
y_0y_1\ , \ y_0y_3\ , \ y_0y_5\ , \ y_2y_4y_6\\
-y_0+y_3-y_4\ , \ -y_0+y_5-y_6\ ,\ -y_0+y_1-y_2
\end{array}
\right\rangle}\\
& \cong &
\frac{\Z[y_0,y_2,y_4,y_6]}{\left\langle 
\begin{array}{c}
y_0^2-y_2^2\ ,\ y_0^2-y_4^2\ ,\ y_0^2-y_6^2\\
y_0^2+y_0y_2\ ,\ y_0^2+y_0y_4\ ,\ y_0^2+y_0y_6\ ,\ y_2y_4y_6\end{array}
\right\rangle}.
\end{eqnarray*}
We next compute the cohomology of $\mathcal{B}$:
\begin{eqnarray*}
H^*(B_0;\Z) & = & \frac{\Z[b_2,b_4,b_6]}{\left\langle b_2b_4b_6\ ,\ b_4-b_6\ ,\ b_2-b_4\right\rangle} \cong 
\frac{\Z[b_2]}{\left\langle b_2^3\right\rangle}; \mbox{ and}
\end{eqnarray*}
\begin{eqnarray*}
H^*(B_1;\Z) & = & \frac{\Z[b_3,b_4,b_5,b_6]}{\left\langle b_3b_4\ ,\ b_5b_6\ ,\ b_3-b_4\ ,\ b_5-b_6\right\rangle} \cong 
\frac{\Z[b_4,b_6]}{\left\langle b_4^2\ ,\ b_6^2\right\rangle}.
\end{eqnarray*}
Putting these two calculations together, and using $x_0$ and $x_1$ as degree zero dummy variable placeholders,
we have
$$
H^*(\mathcal{B};\Z) = x_0\frac{\Z[b_2]}{\left\langle b_2^3\right\rangle} \bigoplus x_1\frac{\Z[b_4,b_6]}{\left\langle b_4^2\ ,\ b_6^2\right\rangle}.
$$
We know that $\widetilde{\varphi}_k$ maps $b_i$ to $y_i$, and $x_i$ to $y_i$.  It is only a matter of bookkeeping to
compute
$$
\begin{array}{rclclcl}
H^0(Y\setminus\mathcal{B};\Z) & = & \Z & & & &  \\
H^1(Y\setminus\mathcal{B};\Z) & = & \ker(\widetilde{\varphi}_1) & = & 0 & & \\
H^2(Y\setminus\mathcal{B};\Z) & = & \coker(\widetilde{\varphi}_1) & = & \spanspan\{ y_4\ ,\ y_6\} & \cong & \Z^2\\
H^3(Y\setminus\mathcal{B};\Z) & = & \ker(\widetilde{\varphi}_2) & = & 0 & & \\
H^4(Y\setminus\mathcal{B};\Z) & = & \coker(\widetilde{\varphi}_2) & = & \spanspan\{ y_4y_6\} & \cong & \Z\\
H^5(Y\setminus\mathcal{B};\Z) & = & \ker(\widetilde{\varphi}_3) & = & \spanspan\{ x_0b_2^2-x_1b_4b_6\} & \cong &\Z \\
H^6(Y\setminus\mathcal{B};\Z) & = & \coker(\widetilde{\varphi}_3) & = & 0. & &
\end{array}
$$

\end{example}

\section{Assembling $H^*(M)$ from $H^*(M\setminus Z)$}\label{sec:coh results}

In this section, we study the cohomology of toric origami manifolds.
There are a number of cases where we can already deduce a number of facts about the cohomology of a toric
origami manifold.  We begin by reviewing these.

The first case is when the template graph is acyclic: this is the topic of or first paper \cite{holm-pires}.
In that case, the cohomology of $M$ is concentrated in even degrees and the equivariant cohomology 
is given by a GKM-type description as detailed in \cite{holm-pires}, or a Stanley-Reisner face ring 
\cite[Theorem~7.7]{masuda-panov}. Furthermore, the ring structure on $H^*(M;\Z)$ can be determined
completely from the discrete geometry of the orbit space, as described in \cite[Corollary~7.8]{masuda-panov}.

The second case where the cohomology ring is determined is when $M$ is prismatic. Then by 
Proposition~\ref{prop:YtimesT2}, $M$ is homeomorphic to $Y\times \T^2$, for a toric symplectic manifold $Y$.
Therefore the cohomology ring is determined by the K\"unneth formula, even over $\Z$ since the cohomology of each
factor is torsion-free.

We now focus on the non-prismatic case, where we can obtain some partial results even for the cyclic case.
We will use a Mayer-Vietoris sequence to obtain the Betti numbers of an arbitrary non-prismatic 
$4$-dimensional toric origami manifold.  
We first note that in general, 
$$H^0(M;\Z)\cong H^{2n}(M;\Z)=\mathbb{Z}$$ 
because $M$ is a connected $2n$-dimensional manifold.
Less trivially, $H^{2n-1}(M;\Z)\cong H_1(M;\Z)$ is the abelianization of $\pi_1(M)$. By Theorem~\ref{thm:fund group}, it is thus $N/N_X\times \mathbb{Z}^\ell$, where $\ell=1+R-L$ is the number of linearly independent cycles in the template graph, which has $L$ vertices and $R$ edges. We also note that the universal coefficients theorem then guarantees that $H^1(M;\Z)\cong \Z^\ell$ and that the torsion in $H^2(M;\Z)$ is precisely $N/N_X$. 

We now proceed with our Mayer-Vietoris sequence.  We enumerate the connected
components $A_1,\dots,A_L$ of $M\setminus Z$ and cover $M$ by open neighborhoods of each $A_i$.  These open 
neighborhoods may be chosen so that they deformation retract to the $A_i$, and their intersections deformation retract onto
components of $Z$.  The Mayer-Vietoris sequence, with integer coefficients, is
\begin{equation}\label{eq:MV}
{\small{ \xymatrix{
0\ar[r]  & H^0(M)\ar[r] & \bigoplus_{i=1}^L H^0(A_i)\ar[r] & H^0(Z)\ar[r] & H^1(M) \ar[r] & \bigoplus_{i=1}^L H^1(A_i) \ar[r] & H^1(Z) \ar `r[d]`[l] `[lllllld]`[d][dlllll]
&\\
& H^2(M)\ar[r] & \bigoplus_{i=1}^L H^{2}(A_i)\ar[r] &H^2(Z)\ar[r] & H^3(M)\ar[r] & \bigoplus_{i=1}^L H^{3}(A_i)\ar[r] &H^3(Z)  \ar `r[d]`[l] `[lllllld]`[d][dlllll]
& \\
& H^4(M) \ar[r] & \bigoplus_{i=1}^L H^4(A_i) \ar[r] & \ldots. &&&
}}}
\end{equation}
The techniques of Section~\ref{sec:coh Y/B} often allow us to compute the ranks of the terms $\bigoplus_{i=1}^L H^k(A_i)$.
The connected components of $Z$ are $\SS^1$-bundles over compact toric symplectic manifolds of dimension $2n-2$.  
In some cases, we may
identify these explicitly.  If we can make both of these computations, it may then be possible to determine the
Betti numbers of $M$.  In particular, when $2n=4$, we may complete all of these steps.  We will also include an
example in dimension $2n=6$.

\begin{theorem}\label{toric origami dim4}
Let $M$ be a $4$-dimensional toric origami manifold. If $M$ is prismatic, then it is homeomorphic to $\SS^2\times\T^2$ and its Betti numbers are
$$b^0(M)=b^4(M)=1 \mbox{ and }b^1(M)=b^2(M)=b^3(M)=2.$$
If $M$ is non-prismatic, let $L$ be the number of vertices and $R$ the number of edges of its template graph, and $M^{\T}$ denote the set of (isolated) fixed points. Then
$$
b^i(M) = \left\{\begin{array}{ll}
1 & i=0\  ,\ 4\\
1+R-L & i = 1\ ,\ 3\\
\# \left(M^{\T}\right)+2R-2L & i =2 
\end{array}\right.
$$
In particular, in both cases the Euler characteristic is $\chi(M) = \# \left(M^{\T}\right)$.
\end{theorem}

\begin{proof}
In the prismatic case, the result is a consequence of Proposition~\ref{prop:YtimesT2} and the K\"unneth formula. In this case, $\chi(M) = \# \left(M^{\T}\right)=0$.

We now turn to the non-prismatic case. Let $X=M/\T$ be the orbit space of $M$.  The fixed points $M^{\T}$ correspond to vertices of $X$ (not to be
confused with vertices of the template graph!).

We first consider the terms $ \bigoplus_{i=1}^L H^*(A_i)$ in \eqref{eq:MV}.
We begin by noting that when $\dim(M) = 4$, Theorem~\ref{prop:prismatic vs not prismatic} determines all the
Betti numbers of each  piece $A_i$.  
Let $P$ be the number of prismatic $A_i$'s. Note that for 2-dimensional polytopes the number of facets
(i.e.\ edges!) equals the number of vertices. A careful application of Theorem~\ref{prop:prismatic vs not prismatic} now gives us:
\begin{align*}
&\sum_{i=1}^Nb^1(A_i)=P\ ;\\
&\sum_{i=1}^Nb^2(A_i)=\#\big\{\text{vertices of $X$}\big\}+2R-2L+P;\\
&\sum_{i=1}^Nb^3(A_i)=2R-L;\\
&\sum_{i=1}^Nb^4(A_i)=0.
\end{align*}

We now turn to the terms $H^*(Z)$ in \eqref{eq:MV}.
When $\dim M = 2n=4$, each $\mathcal{Z}$ is an $\SS^1$-bundle over a toric symplectic 2-sphere, and 
is therefore diffeomorphic to $\SS^1\times \SS^2$, to $\SS^3$ or to a $3$-dimensional lens space $\mathcal{L}=L(p;1)$.
The Betti numbers of these spaces are:
$$b^j(\SS^1\times\SS^2)=\begin{cases}
1 &\mbox{if } j=0,1,2,3 \\
0&\mbox{otherwise} 
\end{cases}
\mbox{\,\,\, and \,\,\,\,}
b^j(\SS^3\mbox{ or }\mathcal{L})=\begin{cases}
1 &\mbox{if } j=0,3 \\
0&\mbox{otherwise} 
\end{cases}.$$

We do know that $b^0(M)=b^4(M) = 1$ and $b^1(M)=b^3(M)=1+R-L$.
Thus, the only group in the sequence \eqref{eq:MV} whose rank we do not know is $H^2(M)$. We proceed by dimension count.
Let $Q$ be the number of  connected components of $Z$ diffeormorphic to $\SS^1\times\SS^2$.  Taking the alternating sum 
of the dimensions of the groups in the Mayer-Vietoris sequence \eqref{eq:MV}, we have
\begin{align*}
& 1-L+R-(1+R-L)+P-Q+b^2(M)-\left(\#\big\{\text{vertices of $X$}\big\}+2R-2L+P\right)+\\
&\hspace{3in}+Q-(1-R+L)+(2R-L)-R+1=0\\
\iff & b^2(M)=\#\big\{\text{vertices of $X$}\big\}+2R-2L,
\end{align*}
completing the proof.
\end{proof}

\begin{remark}\label{non-eg: euler}
An edge (1-dimensional face) of the orbit space $X$ of a toric origami manifold $M$ is 
either is a loop or has two end vertices. In the first case the edge is the moment image 
of a 2-torus, in the second it is the moment image of a sphere, with the end vertices 
being the image of the north and south poles of that sphere. As a consequence, $X$ 
can never have exactly one vertex, and $M$ can never have exactly one fixed point.
Thus, the Euler characteristic of a toric origami manifold cannot be equal to $1$.

The manifold $\C P^2 \# (\SS^1\times\SS^3)$, made up of the building blocks mentioned in Remark~\ref{rmk:type L}, has Euler characteristic 
$$\chi\left( \C P^2 \# (\SS^1\times\SS^3) \right)=\chi(\C P^2)+\chi(\SS^1\times\SS^3)-\chi(\SS^4)=3+0-2=1$$
and therefore does not admit a toric origami structure.
\end{remark}

\begin{remark}
The second Betti number of a toric origami manifold bears a resemblance to that of a toric symplectic manifold.  We have
just seen that for a $4$-dimensional toric origami manifold $M$, setting $\ell=1+R-L$,
$$
b^2(M) = \# \big\{\mbox{vertices in } M/\T \big\} -2+2\ell\ .
$$
For a toric symplectic manifold $Y$ of dimension $2n$,
\begin{equation}\label{eq:vertices classic}
b^2(Y) = \# \big\{ \mbox{facets in } Y/\T \big\} -n\ .
\end{equation}
In dimension $4$, we can rewrite \eqref{eq:vertices classic} as
$$b^2(Y)= \# \big\{ \mbox{vertices in } Y/\T \big\} -2\ .$$
Thus, these two descriptions are the same, up to a correction for the rank $\ell$ of $\pi_1(M/\T)$.
\end{remark}

\begin{example}
Let $M_1$ be the toric origami manifold described in Figure~\ref{fig:4dim example}, left and center. This information completly determines the template and therefore the manifold. Its template graph has 4 vertices and 4 edges, and the manifold has 4 fixed points. Using Theorem~\ref{toric origami dim4} we conclude that the Betti numbers of this manifold are:
$$b^0(M_1)=b^1(M_1)=b^3(M_1)=b^4(M_1)=1\mbox{ and } b^2(M_1)=4.$$

\begin{center}
\begin{figure}[h]
\includegraphics[height=1in]{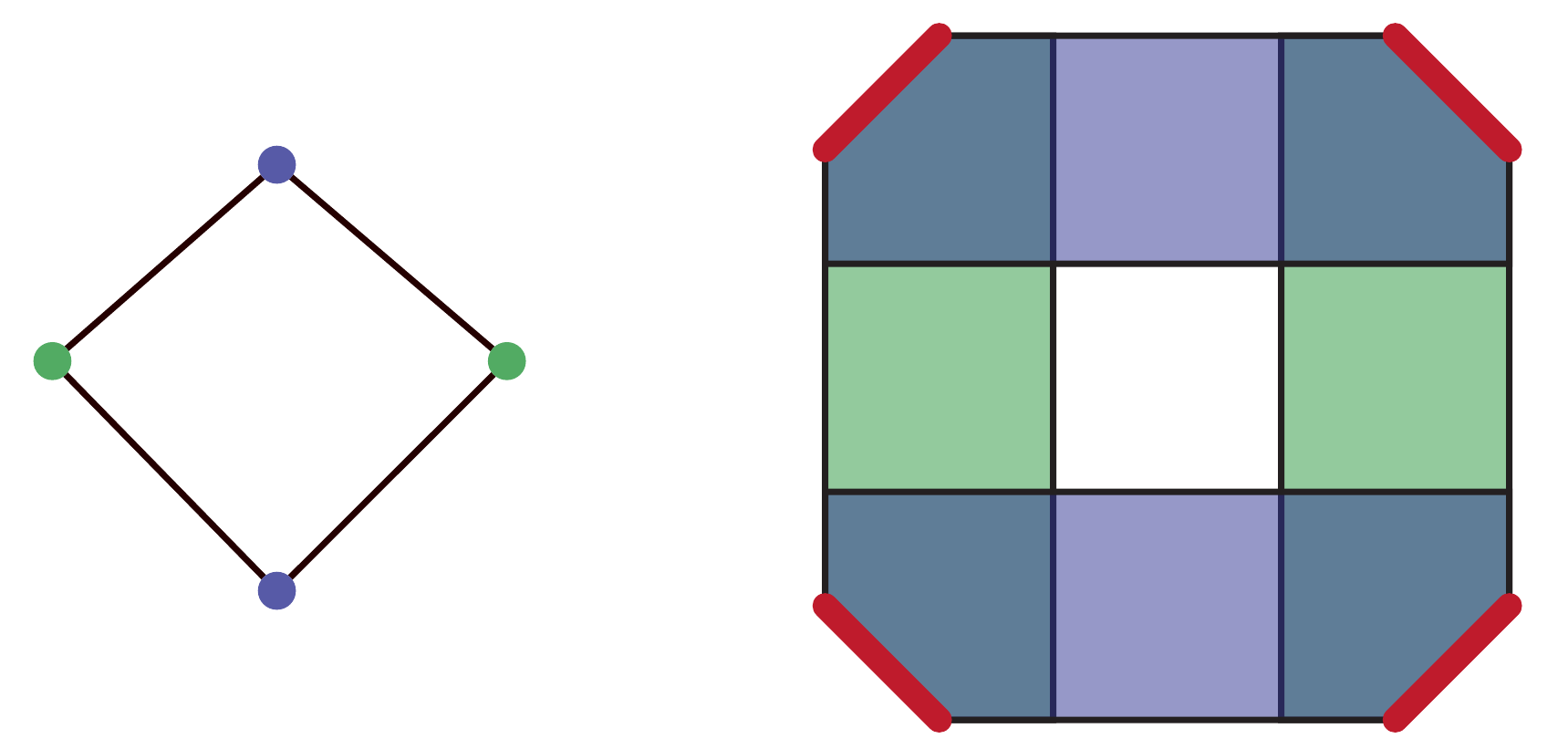} \hskip 0.3in
\includegraphics[height=1in]{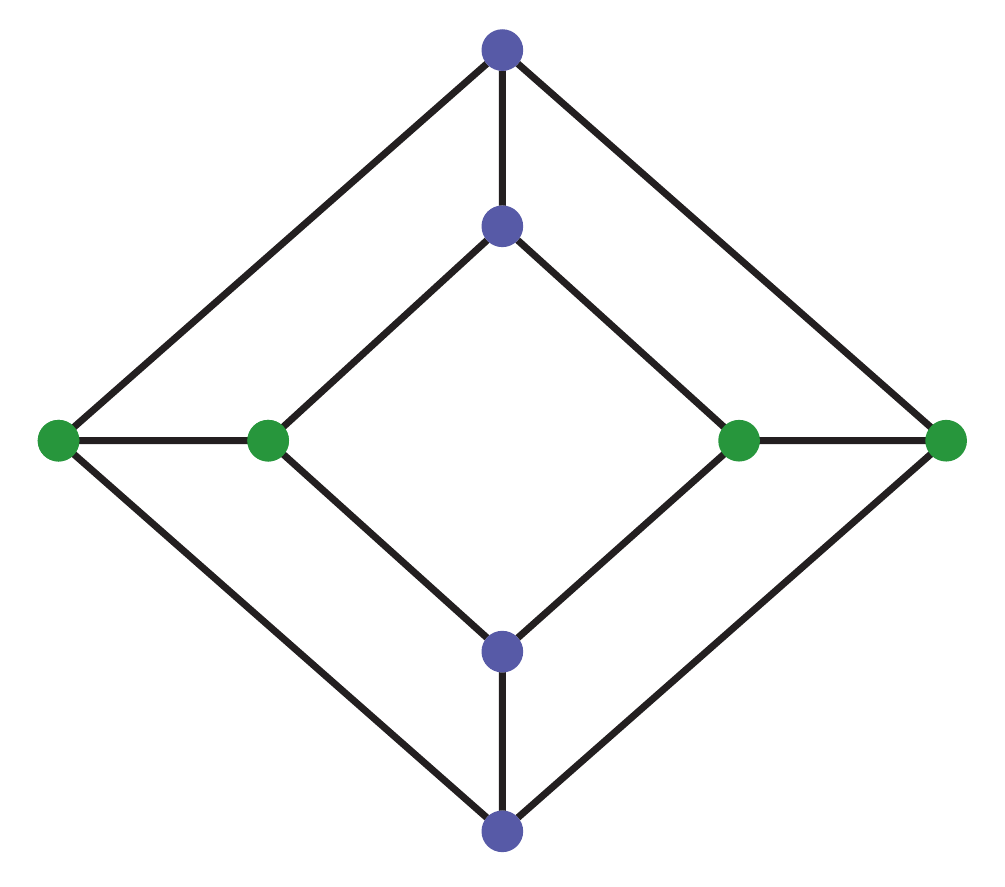} 
\caption{{\bf Left}: the template graph of a toric origami manifold $M_1$.\\ 
{\bf Center}: the moment image of the toric origami manifold $M_1$.\\ 
{\bf Right}: the template graph of a toric origami manifold $M_2$ obtained from taking two copies of $M_1$ 
and gluing their orbit spaces along 4 pairs of non-folded facets.}
\label{fig:4dim example}
\end{figure}
\end{center}

\noindent The orbit space of $M_1$ has 4 non-folded facets, each corresponding to a symplectic 2-sphere embeded in $M_1$.  Let $M_2$ be the toric origami manifold obtained by taking two copies of $M_1$ and gluing them together along each of the 4 pairs of symplectic 2-spheres with the same moment image. The resulting template graph is on the right hand side of Figure~\ref{fig:4dim example} and has 8 vertices and 12 edges. The vertices and edges that appear in each of the two concentric square rings of this template graph correspond to the two copies of $M_1$, the remaining 4 edges in the template graph correspond to the new connected components of the fold. The manifold $M_2$ thus created has no fixed points. Using Theorem~\ref{toric origami dim4} we obtain its Betti numbers:
$$b^0(M_2)=b^4(M_2)=1, \,\, b^1(M_2)=b^3(M_2)=5, \,\,b^2(M_2)=8.$$
\end{example}

\begin{example}
We now turn to a higher dimensional example for which the computations are still tractable and for which we can obtain all the Betti numbers. Let $M$ be obtained from two copies of the manifold examined in Example~\ref{ex:truncated cube}, glued together along the two agreeing pairs of facets marked in red.

\begin{center}
\begin{figure}[h]
\includegraphics[height=1.8in]{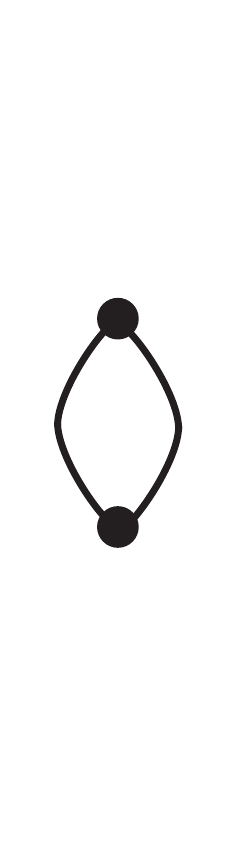} \hskip 0.8in
\includegraphics[height=1.8in]{truncatedcubewithlabels} 
\caption{Left: the template graph of the manifold $M$. Right: each vertex of the template correponds to a copy of the toric symplectic manifold with moment image a truncated cube (the same as in Figure~\ref{fig:truncated cube}). One edge of the template graph corresponds to gluing together the pair of facets $F_0$, the other edge to gluing togehter the pair of facets $F_1$.}
\label{fig:double truncated cube}
\end{figure}
\end{center}

Most of the terms in the Mayer Vietoris sequence with integer coefficients \eqref{eq:MV} are known, the $\bigoplus_{i=1}^L H^*(A_i)$ terms from Example~\ref{ex:truncated cube} and the $H^*(Z)$ terms from direct computation. Indeed, $Z$ is the disjoint union $Z=\SS^5 \sqcup  \left(\SS^2\times\SS^2\times\SS^1\right)$, the first with moment moment image the facet $F_0$ and the second with moment image the facet $F_1$ in Figure~\ref{fig:double truncated cube}, and therefore:
$$H^k(Z;\Z)=\begin{cases} 
\Z^2 &\mbox{for } k=0,2,3,5 \\ 
\Z & \mbox{for } k= 1,4.
\end{cases} $$
Furthermore, we know that $H^0(M;\Z)=H^6(M;\Z)=\Z$ because $M$ is a 6-dimensional connected manifold and that $H^1(M;\Z)=\Z$ and $H^5(M;\Z)=\Z$ because $\pi_1(M)=\Z$.
Taking an alternating sum of the ranks of the groups in the 
sequence \eqref{eq:MV}, we obtain the remaining Betti numbers of $M$:
$$b^0(M)=b^1(M)=b^3(M)=b^5(M)=b^6(M)=1 \mbox{ and } b^2(M)=b^4(M)=2.$$
\end{example}



\end{document}